\pdfoutput=1
\RequirePackage{ifpdf}
\ifpdf 
\documentclass[pdftex]{sigma}
\else
\documentclass{sigma}
\fi

\usepackage{tikz}
\usetikzlibrary{arrows}

\usepackage[all]{xy}
\usepackage{pb-diagram,pb-xy}

\usepackage{stmaryrd}
\usepackage{enumitem}

{\theoremstyle{definition}
\newtheorem{question}[theorem]{Question}
}

\newcommand{\comment}[1]{} 
\newcommand{\mt}[1]{{\text{\rm #1}}} 

\newcommand{\rh}{\hookrightarrow}

\newcommand{\boot}{\mathrm{boot}}
\newcommand{\con}{\mathrm{con}}
\newcommand{\pre}{\mathrm{pre}}
\newcommand{\std}{\mathrm{std}}
\newcommand{\Estd}{\mathrm{Estd}}
\newcommand{\tor}{\mathrm{torp}}

\newcommand{\att}{\mathrm{att}}
\newcommand{\spin}{\mathrm{spin}}

\newcommand{\MSpin}{\mathsf{MSpin}}

\newcommand{\Id}{\mathrm{Id}}

\newcommand{\fb}{\mbox{-}\mathrm{fb}}

\newcommand{\psc}{\mathrm{psc}}
\newcommand{\Z}{\mathbb{Z}} 
\newcommand{\CP}{\mathbb{CP}} 
\newcommand{\HP}{\mathbb{HP}}

\renewcommand{\max}{\operatorname*{\mt{max}}}
\renewcommand{\min}{\operatorname*{\mt{min}}}
\renewcommand{\lim}{\operatorname*{\mt{lim}}}

\renewcommand{\dim}{\operatorname{\mt{dim}}}

\newcommand{\Riem}{{\mathcal R}}

\newcommand{\p}{\partial}

\newcommand{\bord}{\rightsquigarrow}

\numberwithin{equation}{section}

\newtheorem{Theorem}{Theorem}[section]
\newtheorem{Corollary}[Theorem]{Corollary}
\newtheorem{Conjecture}[Theorem]{Conjecture}
\newtheorem{Lemma}[Theorem]{Lemma}
\newtheorem{Proposition}[Theorem]{Proposition}
 { \theoremstyle{definition}
\newtheorem{Definition}[Theorem]{Definition}
\newtheorem{Remark}[Theorem]{Remark} }

\begin{document}

\newcommand{\arXivNumber}{2005.03073}

\renewcommand{\thefootnote}{}

\renewcommand{\PaperNumber}{034}

\FirstPageHeading

\ShortArticleName{Homotopy Invariance of the Space of Metrics with Positive Scalar Curvature}

\ArticleName{Homotopy Invariance of the Space of Metrics\\ with Positive Scalar Curvature on Manifolds\\ with Singularities\footnote{This paper is a~contribution to the Special Issue on Scalar and Ricci Curvature in honor of Misha Gromov on his 75th Birthday. The full collection is available at \href{https://www.emis.de/journals/SIGMA/Gromov.html}{https://www.emis.de/journals/SIGMA/Gromov.html}}}

\Author{Boris BOTVINNIK~$^{\rm a}$ and Mark G.~WALSH~$^{\rm b}$}

\AuthorNameForHeading{B.~Botvinnik and M.~Walsh}

\Address{$^{\rm a)}$~Department of Mathematics, University of Oregon, Eugene, OR, 97405, USA}
\EmailD{\href{mailto:botvinn@uoregon.edu}{botvinn@uoregon.edu}}
\URLaddressD{\url{http://darkwing.uoregon.edu/~botvinn/}}

\Address{$^{\rm b)}$~Department of Mathematics and Statistics, Maynooth University, Maynooth, Ireland}
\EmailD{\href{mailto:Mark.Walsh@mu.ie}{Mark.Walsh@mu.ie}}

\ArticleDates{Received June 16, 2020, in final form March 24, 2021; Published online April 02, 2021}

\Abstract{In this paper we study manifolds, $X_{\Sigma}$, with fibred singularities, more specifically, a relevant space ${\mathcal R}^{\rm psc}(X_{\Sigma})$ of Riemannian metrics with positive scalar curvature. Our main goal is to prove that the space ${\mathcal R}^{\rm psc}(X_{\Sigma})$ is homotopy invariant under certain surgeries on~$X_{\Sigma}$.}


\Classification{53C27; 57R65; 58J05; 58J50}


\renewcommand{\thefootnote}{\arabic{footnote}}
\setcounter{footnote}{0}

\section{Introduction}\label{Intro}
\subsection{Existence of a psc-metric}
A classical result in this subject concerns the existence of metrics
of positive scalar curvature (psc-metrics) on a simply-connected
smooth closed manifold $M$. There are two cases here: either~$M$ is a
spin manifold or it is not. Recall that in the case when $M$ is spin,
there is an~index $\alpha(M)\in KO^{-n}$ of the Dirac operator valued
in real $K$-theory. Here is the result:
\begin{Theorem}[Gromov--Lawson~\cite{GL}, Stolz~\cite{St}]\label{thm:GL-Stolz}
Let $M$ be a smooth closed simply connected manifold of dimension $n\geq 5$.
\begin{enumerate}\itemsep=0pt
\item[$(i)$] If $M$ is spin, then $M$ admits a psc-metric if and only
 if the index $\alpha(M)\in KO^{-n}$ of the Dirac operator on $M$
 vanishes.
\item[$(ii)$] If $M$ is not spin, then $M$ always admits a psc-metric.
\end{enumerate}
\end{Theorem}
We denote by $\Riem^{\psc}(M)$ the space of psc-metrics on $M$. Recall
that one of the major tools used to prove Theorem~\ref{thm:GL-Stolz}
is the surgery technique due to Gromov and Lawson (proved
independently by Schoen and Yau). In particular, Gromov--Lawson
observed that a psc-metric survives surgeries of codimension at least
three (such surgeries are called \emph{admissible}). It turns out that
the homotopy type of the space $\Riem^{\psc}(M)$ is invariant under
such surgeries; see~\cite{Ch,EF,W1}.

\subsection{Existence of a psc-metric on a manifold with Baas--Sullivan
 singularities}
We start with the simplest case, where the geometrical picture is
transparent. Let $(L, g_L)$ be a~closed Riemannian manifold, in which
the metric $g_L$ is assumed to have zero scalar curvature. Let $Y$ be
a closed smooth manifold, such that the product $Y\times L$ is a
boundary of a smooth manifold $X$, i.e., $\p X= Y\times
L$. Here is a natural geometrical question:
\vspace{1mm}

\begin{question}
Does there exist a psc-metric $g_Y$ on $Y$, such
 that the product metric $g_Y+ g_L$ on~$\p X = Y\times L$ can be
 extended (being a product near $\p X$) to a psc-metric
 $g_X$ on $X$?
 \end{question}

It is convenient to denote $\beta X := Y$ (which is sometimes called
the {\em Bockstein manifold of~$X$}). Thus, for now, the boundary of~$X$ forms the total space of a trivial bundle,
\begin{gather*}
\p X=\beta X\times L\rightarrow \beta X.
\end{gather*}
In this case, a {\em manifold with
Baas--Sullivan singularities of the type $L$ $($or just
$L$-singularities$)$}, and denoted $X_{\Sigma}$, is obtained as
\begin{gather*}
 X_{\Sigma}:= X\cup_{\p X} -\beta X \times C(L),
\end{gather*}
where $C(L)$ is a cone over $L$ (and the minus sign represents a
change of orientation in the case when manifolds are oriented).

Assume for a moment that there is a psc-metric $g_X$ on $X$ such that
$g_X$ is a product metric of~the form $g_{\p X}+{\rm d}t^{2}$, near $\p X$,
where the boundary metric $g_{\p X}:=g_X|_{\p X}$ satisfies $g_{\p
 X}=g_{\beta X}+ g_L$. Then we would like to extend the metric
$g_{X}$ to a psc-metric (with singularities) on $X_{\Sigma}$. (For
brevity we will use the term metric here with the understanding that
on $X_{\Sigma}$ (or~on~$C(L)$), metrics are necessarily singular.) To
do this, we begin by extending the scalar-flat metric, $g_{L}$, to a
scalar-flat metric on the cone $C(L)$, which we denote
$g_{C(L)}$. There is an obvious way to do this and details are
provided in the appendix.

To obtain the desired metric on $X_{\Sigma}$, it will be necessary to
attach this conical metric, ${g_{C(L)}}$, to~a~cylindrical metric
${\rm d}t^{2}+g_{L}$. While these metrics do not attach smoothly, it is
possible to~bend one end of the cylinder inwards so that it matches
the conical metric. This process creates some positive but no negative
scalar curvature. With this in mind, we construct an~``attaching"
metric of non-negative scalar curvature on the cylinder $X\times
[0,1]$ which takes the form ${\rm d}t^{2}+g_{L}$ at one end and smoothly
attaches to ${g_{C(L)}}$ at the other. This metric is denoted
$g_{\att(L)}$ and~details of~its construction are given in the
appendix.

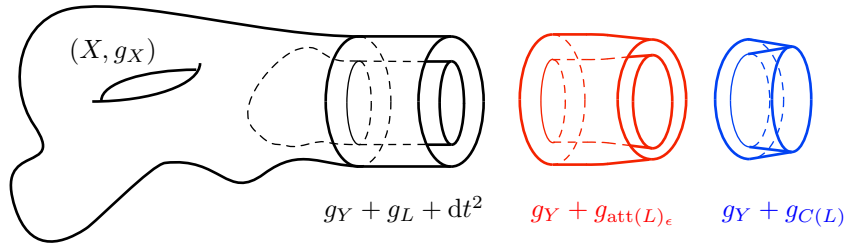
\begin{figure}[!htbp]
\begin{center}
\vspace{5mm}
\begin{picture}(0,0)
\includegraphics{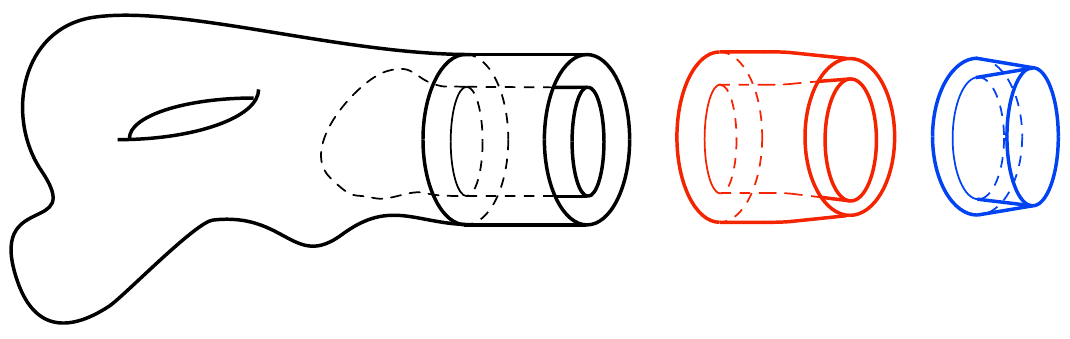}%
\end{picture}
\setlength{\unitlength}{3947sp}
\begingroup\makeatletter\ifx\SetFigFont\undefined%
\gdef\SetFigFont#1#2#3#4#5{%
 \reset@font\fontsize{#1}{#2pt}%
 \fontfamily{#3}\fontseries{#4}\fontshape{#5}%
 \selectfont}%
\fi\endgroup%
\begin{picture}(5079,1559)(1902,-7227)
\put(3800,-7000){\makebox(0,0)[lb]{\smash{{\SetFigFont{10}{8}{\rmdefault}{\mddefault}{\updefault}{\color[rgb]{0,0,0}$g_{Y}+g_{L}+{\rm d}t^{2}$}%
}}}}
\put(5100,-7000){\makebox(0,0)[lb]{\smash{{\SetFigFont{10}{8}{\rmdefault}{\mddefault}{\updefault}{\color[rgb]{1,0,0}$g_{Y}+g_{\att(L)_{\epsilon}}$}%
}}}}
\put(6300,-7000){\makebox(0,0)[lb]{\smash{{\SetFigFont{10}{8}{\rmdefault}{\mddefault}{\updefault}{\color[rgb]{0,0,1}$g_{Y}+g_{C(L)}$}%
}}}}
\put(2200,-6000){\makebox(0,0)[lb]{\smash{{\SetFigFont{10}{8}{\rmdefault}{\mddefault}{\updefault}{\color[rgb]{0,0,0}$(X, g_{X})$}%
}}}}
\end{picture}%
\end{center}\vspace*{-2mm}
\caption{Gluing together the components, $g_X$, $g_{Y}+g_{\att(L)}$ and
 $g_{Y}+g_{C(L)}$ of the metric $g_{X_\Sigma}$.}
\label{conemetric}
\end{figure}
\begin{Remark}\label{rem:well-adp}
 We emphasize that, according to our construction, the metrics
 $g_{\att(L)}$ and ${g_{C(L)}}$ are both invariant under the action
 of the isometry group $\mathrm{Iso}(g_L)$ of the metric $g_L$. In
 our work, the role of $L$ will be played by either the circle $S^1$
 (with $U(1)$ as the isometry group), some other appropriate Lie
 group or a homogeneous space, such as $S^n$ or $\HP^n$, equipped
 with the natural metric $g_L$ of non-negative constant scalar
 curvature.
\end{Remark}
Returning now to the manifold $X_\Sigma$, we
define a metric $g_{X_\Sigma}$ by
\begin{gather*}
g_{X_{\Sigma}}:= g_X\cup (g_{Y}+g_{\att(L)})\cup (g_Y+g_{C(L)}),
\end{gather*}
obtained by gluing in the obvious way; see Figure~\ref{conemetric}.
The following fact is an easy consequence of Proposition~\ref{attachprop}.
\begin{Proposition}
The construction of $g_{X_{\Sigma}}$ above yields a smooth $($singular$)$
metric of non-negative scalar curvature on~$X_{\Sigma}$.
\end{Proposition}
The above construction motivates the notion of a well-adapted metric
on $X_{\Sigma}$. Before stating~it, we fix a scalar-flat metric $g_L$
on $L$ and equip $X$ with a boundary collar $c\colon\p X\times
[-1,1]\rightarrow X$, where $c(\p X\times \{-1\})=\p X$.

We say that a metric $g$ on $X_{\Sigma}$ is a
\emph{well-adapted Riemannian metric on $X_{\Sigma}$} if
\begin{enumerate}\itemsep=0pt
\item[$(i)$] the restriction $g|_{X}$ satisfies the following conditions
 on the collar $c$:
\begin{gather*}
c^{*}g|_{\p X\times [0,1]}=g_{\beta X}+g_{L}+{\rm d}t^{2} \qquad\text{and}\qquad
c^{*}g|_{\p X\times [-1,0]}=g_{\beta X}+g_{\att(L)},
\end{gather*}
\item[$(ii)$] the restriction $g|_{\beta X \times C(L)}$ splits as a
 product-metric $g|_{\beta X \times C(L)} = g_{\beta X}+ g_{C(L)}$.
\end{enumerate}
We denote by $\Riem(X_{\Sigma})$ the space of all well-adapted
Riemannian metrics on $X_{\Sigma}$, and by $\Riem^{\psc}(X_{\Sigma})$
its subspace of psc-metrics. Thus, Question 1 above is equivalently
asking whether the space $\Riem^{\psc}(X_{\Sigma})$ is non-empty.
This existence question was addressed and even affirmatively resolved
for some particular examples of the singularity types $L$ (provided
that all manifolds involved are spin and both $X$ and $\beta X$ are
simply-connected; see~\cite{B1}).

There is a particularly interesting example here.
Let us consider spin manifolds, and choose $L=S^1$ with a non-trivial
spin structure, so that $L$ represents the generator \mbox{$\eta\in
\Omega^{\spin}_1=\Z_2$}. We~denote by $\Omega^{\spin,\eta}_*(-)$ the
bordism theory of spin manifolds with $\eta$-singularities, and by~$\MSpin^{\eta}$ the corresponding representing spectrum. We~refer to~\cite{B1}
for details on cobordism theory $\Omega^{\spin,\eta}_*(-)$.

It turns out, there exists a Dirac operator on spin-manifolds with
$\eta$-singularities. Furthermore, there is a natural transformation
$\alpha^{\eta}\colon \Omega^{\spin,\eta}_*\to KO_*^{\eta}$ which evaluates
the index of that Dirac operator, where the ``$K$-theory with
$\eta$-singularities'' $KO_*^{\eta}(-)$ coincides with the usual complex
$K$-theory. Here
is the result from~\cite{B1}:
\begin{Theorem}[Botvinnik~\cite{B1}]\label{Th1}
	Let $X$ be a simply connected spin manifold with nonempty
 $\eta$-singularity of dimension $n\geq 7$. Then $X$ admits a
 metric of positive scalar curvature if and only if
 $\alpha^{\eta}([X])=0$ in the group $KO_n^{\eta} \cong KU_n$.
\end{Theorem}
\subsection{Existence of a psc-metric on a manifold with fibred
 singularities} There are more general objects, namely, ``manifolds
with fibred singularities'' (or pseudo-mani\-folds with singularities of
depth one). Here again, we start with a manifold $X$ with boundary $\p
X\neq \varnothing$, which is the total space of a fibre bundle $\p X \to
\beta X$ with the fibre $L$. Usually the manifold $L$ is referred to
as the \emph{link}. To get \emph{geometrically interesting} objects,
we fix the following data on the manifold $L$:
\begin{enumerate}\itemsep=0pt
\item[$(a)$] a metric $g_L$ on $L$ of non-negative constant scalar
 curvature,
\item[$(b)$] a subgroup $G$ of the isometry group $\mathrm{Isom}(g_L)$
 of the metric $g_L$.
\end{enumerate}
In particular, if $G$ is a nontrivial subroup, the
 metric $g_L$ is assumed to be homogeneous. Henceforth, we assume
 that the bundle $\p X \to \beta X$ is a $G$-bundle; thus
the bundle $\p X \to \beta X$ is induced by a structure map $f\colon \beta
X \to BG$.

Let $C(L)$ be a cone over $L$ with a cone metric, $g_{C(L)}$, which
restricts to $g_{L}$ on the base and is scalar-flat. As~we mentioned
in Remark~\ref{rem:well-adp}, the attaching metric $g_{\att(L)}$ and
the cone metric~${g_{C(L)}}$ are both invariant under the action of
the isometry group $\mathrm{Isom}(g_L)$. In particular, the isometry
action of $G$ on $L$ extends automatically to $C(L)$ and this gives
rise to a fibre bundle $N(\beta X) \to \beta X$, obtained by
``inserting'' the cone $C(L)$ as a fibre with the same structure group~$G$. Thus, we obtain a {\em manifold with fibred singularities},
$X_{\Sigma}$, given as $X_{\Sigma}:=X\cup_{\p X} -N(\beta X)$. Such a
manifold is referred to as an {\em $(L,G)$-manifold.}

Now we consider spin $(L,G)$-manifolds $X_{\Sigma}:=X\cup_{\p X}
-N(\beta X)$; this means that we first fix a~spin structure on $L$,
and $X$ and $\beta X$ are given spin structures which are respected by
the bundle map $\p X\to \beta X$. We~will assume that the isometry
group $G$ preserves the fixed spin structure on $L$. For~short, we
call such a bundle map $\p X\to \beta X$ a \emph{geometric
 $L$-bundle}.

There is a natural concept of spin cobordism
 between $(L,G)$-manifolds. Namely, we say that $W_{\Sigma} \colon X_{\Sigma}\bord
 X_{\Sigma,1} $ is a \emph{spin $(L,G,\fb)$-cobordism between}
\begin{gather*}
 X_{\Sigma}=X\cup_{\p X} -N(\beta X) \qquad \mbox{and} \qquad X_{\Sigma,1}=X_1\cup_{\p
 X_1} -N(\beta X_1),
\end{gather*}
if $W_{\Sigma}:=W \cup_{\p^s W} -N(\beta
 W)$, where the boundary $\p W$ is decomposed as
\begin{gather*}
\p W = X \cup \p^s W \cup X_1,
\end{gather*}
and $\p^s W\colon \p X \bord \p X_1$ is a spin cobordism over a geometric
$L$-bundle map $p\colon \p^s W \to \beta W$ which restricts to given
geometric $L$-bundles $\p X \to \beta X$ and $\p X_1 \to \beta
X_1$. This gives a corresponding cobordism theory $\Omega_n^{\spin,
 (L,G)\fb}(\mbox{-})$; see~\cite{BPR,BPR2, BR} for more details.

It turns out there are several interesting cases when the Dirac
operator is well defined on~such spin $(L,G)$-manifolds, and the index
of such operator gives a homomorphism
\begin{gather*}
\alpha^{(L,G)\fb}\colon\ \Omega_n^{\spin, (L,G)\fb} \to KO_n^{(L,G)\fb},
\end{gather*}
where $KO_n^{(L,G)\fb}$ is an appropriate $K$-theory; see
\cite{BPR,BPR2,BR} and the examples below in this section.

We have to modify the above definition of well-adapted Riemannian
metric on $X_{\Sigma}$ (we give a~detailed definition in Section~\ref{prim}). Roughly, in this setting, a \emph{well-adapted
 Riemannian metric}~$g$ on $X_{\Sigma}=X\cup_{\p X} -N(\beta X)$ is a
regular Riemannian metric restricted to $X$ (satisfying certain
product conditions near the boundary), and the restriction metric
$g|_{N(\beta X)}$ is determined by~a~requi\-re\-ment that the projection
$N(\beta X) \to \beta X$ is a Riemannian submersion (which has a~structure group $G \subset \mathrm{Isom}(g_L)$) and with the cone
metric $g_{C(L)}$ on the fibre. We~denote by~$\Riem(X_{\Sigma})$ the
space of well-adapted Riemannian metrics on $X_{\Sigma}$, and by
$\Riem^{\psc}(X_{\Sigma})$ its subspace of~psc-metrics; see Section~\ref{welladaptsec} for a precise definition of this subspace. Below
we describe some interesting cases.

\medskip\noindent
{\bf 1.3.1.}
We assume that all $(L,G)$-manifolds are spin,
$L=S^1$ representing $\eta\in \Omega_1^{\spin}$, and $G=S^1$. We~obtain a corresponding bordism group $\Omega^{\spin,\eta\fb}_*$ of
such manifolds. Then there exists an~appropriate Dirac operator on
$X_{\Sigma}$, and index map $\alpha^{\eta\fb}\colon
\Omega^{\spin,\eta\fb}_*\to KO^{\eta\fb}_*$ evaluating the index of
that Dirac operator. Here is the existence result for psc-metrics in
that setting:

\begin{Theorem}[Botvinnik--Rosenberg~\cite{BR}]
Let $X_{\Sigma}= X\cup N(\beta X)$ be a simply connected spin manifold
with fibred $\eta$-singularity $($i.e., $X$ and $\beta X$ are
simply-connected and spin$)$ of dimension $n\geq 7$. Then
$\Riem^{\psc}(X_{\Sigma})\neq\varnothing $ if and only if
$\alpha^{\eta\fb}([X_{\Sigma}])=0$ in the group $KO^{\eta\fb}_n$.
\end{Theorem}

\medskip\noindent
{\bf 1.3.2.} Again,
we assume that all $(L,G)$-manifolds are spin: here $L$ is equipped
with a metric $g_L$ with constant scalar curvature $s_L=\ell(\ell-1)$,
$\dim L=\ell$, and $G$ is a subgroup of the isometry group of the
metric $g_L$. We~assume $L=\p \bar L$, where $\bar{L}$ is a smooth
compact manifold with boundary, and the $G$-action on $L$ extends to a
$G$-action on $\bar L$. In this setting, an $(L,G)$-manifold
$X_{\Sigma}$ could be described as a triple $(X,\beta X, f)$, where
$X$ is a manifold with boundary $\p X$, which is a total space of a
fibre bundle $\p X\to \beta X$ (with a fibre $L$ and a structure group
$G$) given by a map $f\colon \beta X \to BG$. In this setting we have
indices $\alpha(\beta X)\in KO_{n-\ell-1}$ and
$\alpha_{\mathrm{cyl}}(X)\in KO_{n}$, where $n=\dim X$.

Here are the existence results:
\begin{Theorem}[Botvinnik--Piazza--Rosenberg~\cite{BPR}]
\label{thm:Lbdy}
Let $(X, \partial X, f)$ define a closed $(L,G)$-singular spin
manifold $X_\Sigma$. Assume that $X$, $\beta X$, and $G$ are all
simply connected, that $n-\ell\ge 6$, and~sup\-pose that $L$ is a spin
boundary, say $L=\partial\bar L$, with the standard metric $g_L$ on
$L$ extending to~a~psc-metric on $\bar L$, and with the $G$-action on
$L$ extending to~a~$G$-action on $\bar L$. Assume that the two
obstructions $\alpha(\beta X)\in KO_{n-\ell-1}$ and $\alpha_{{\rm
 cyl}}(X)\in KO_n$ both vanish. Then $X_\Sigma$ admits a~well-adapted psc-metric.
\end{Theorem}

\begin{Theorem}[Botvinnik--Piazza--Rosenberg~\cite{BPR}]
\label{thm:BaasHP2}
Let $(X, \beta X, f)$ define a closed $(L,G)$-singular spin manifold
$X_\Sigma$, with $L=\mathbb{HP}^{2k}$ and $G={\rm Sp}(2k+1)$, $n\ge 1$. Assume
that $\p X = \beta X \times L$, i.e., the $L$-bundle over $\beta X$ is
trivial, or in other words that the singularities are of Baas--Sullivan
type. Then if $X$ and $\beta X$ are both simply connected and
$n-8k\ge 6$, $X_\Sigma$ has an adapted psc-metric if and only if the
$\alpha$-invariants $\alpha(\beta X) \in KO_{n-8k-1}$ and
$\alpha_{{\rm cyl}}(X) \in KO_n$ both vanish.
\end{Theorem}
There are several other interesting cases and also more general
results when $X_{\Sigma}$ has non-trivial fundamental group; see
\cite{BPR2}. Now we are ready to address our main result concerning
homotopy invariance of the corresponding spaces of psc-metrics on
$X_{\Sigma}$.

\subsection{Main result}
The homotopy-invariance of certain spaces of psc-metrics is a crucial
property which has allowed detection of their non-trivial homotopy groups.
Let $M$ be a closed spin manifold. An important
consequence of the results due to Chernysh~\cite{Ch}, Walsh
\cite{W1,W3,W2} (see also recent work by Ebert and Frenck~\cite{EF})
is that the homotopy type of the space $\Riem^{\psc}(M)$ is an invariant
of the bordism class $[M]\in \Omega^{\spin}_n$ (provided $M$ is
simply-connected and $n\geq 5$).\footnote{There is also a similar
 result for non-simply connected manifolds.}

Notice that if $X_{\Sigma}=X\cup_{\p X} -N(\beta X)$ is a
pseudo-manifold with $(L,G)$-singularities equipped with structure
map $f\colon \beta X \to BG$, then there are two types of surgery possible on $X$:
\begin{enumerate}\itemsep=0pt
\item[$(i)$] a surgery on its ``resolution'', i.e., the interior
 $X\subset X_{\Sigma}$ away from the boundary $\p X$,
\item[$(ii)$] a surgery on the structure map $f\colon \beta X \to BG$.
\end{enumerate}
In case $(i)$ all constructions are the same as in the case of
closed manifolds; however, in case $(ii)$, we have to be more
careful. Indeed, let $\bar B\colon \beta X \bord \beta X_1$ be the trace
of a surgery on the map $f\colon \beta X \to BG$, with $\p \bar B = \beta
X \sqcup - \beta X_1$ and $\beta X_1$, the manifold
 obtained from $\beta X$ by surgery on~$f$. Then the map $f$ extends
to a map $\bar f\colon \bar B\to BG$ which gives a fibre bundle $\bar p\colon
Z\to \bar B$ with the fibre $L$.
This gives us a new manifold $X_1= X\cup_{\p X} Z$ with boundary $\p
X_1$, the total space over a new Bockstein $\beta X_1$ with the same
fibre $L$. Also we obtain a new conical part $N(\beta X_1)$ as
above. All of this results in a new pseudo-manifold
\begin{gather}\label{eq1a}
X_{\Sigma,1} = X_1\cup_{\p X_1}- N(\beta X_1), \qquad X_1= X\cup_{\p X} Z,
\end{gather}
with structure map $f_1= \bar f|_{\beta X_1}\colon \beta X_1 \to BG$.
Here is our main technical result:
\begin{Theorem}\label{thmA}
Let $X_{\Sigma}=X\cup_{\p X} -N(\beta X)$ be a pseudo-manifold with
$(L,G)$-singularities, with $\dim X = n$, $\dim L=\ell$.
\begin{enumerate}\itemsep=0pt
\item[$(i)$] Let $i\colon S^p\subset X$ be a sphere with trivial normal bundle,
 and $X_{\Sigma}'$ be the result of surgery on $X_{\Sigma}$ along
 $S^p$. Then if $2\leq p\leq n-3$, the spaces $\Riem^{\psc}(X_{\Sigma})$ and
$\Riem^{\psc}(X_{\Sigma}')$ are weakly homotopy equivalent.
\item[$(ii)$] Let $i\colon S^p\subset \beta X$ be a sphere with trivial
 normal bundle, and with $f\circ \iota\colon S^p \to BG$ homotopic to
 zero. Let $\bar B$ be a trace of the surgery along $S^p\subset \beta X$
 with $\p \bar B = \beta X\sqcup -\beta X_1$ and a~structure map $\bar f\colon
 \bar B\to BG$. Then if $2\leq p \leq n-\ell-3$, the spaces
 $\Riem^{\psc}(X_{\Sigma})$ and $\Riem^{\psc}(X_{\Sigma,1})$ are
 homotopy equivalent, where $X_{\Sigma,1}$ is given by~\eqref{eq1a}.
 \end{enumerate}
\end{Theorem}
\begin{Remark}\label{Part(i)A}
As it deals with surgery on the interior, part~$(i)$ of Theorem~\ref{thmA}
above follows directly from Chernysh's theorem~\cite{Ch}.
Our contribution in this paper is in proving part~$(ii)$.
\end{Remark}
Theorem~\ref{thmA} could be applied to a variety of interesting examples.
Among these are:
\begin{enumerate}\itemsep=0pt
\item[{\rm (1)}] Let $L=\left<k\right>$ be the set of $k$ points, and
 let $G=\Z_k$ be its ``isometry group''. Then a~$(\left<k\right>\!,\Z_k)$-manifold $X_{\Sigma}=X\cup_{\p X} -N(\beta
 X)$ is assembled out of a manifold $X$ with boundary $\p X$ equipped
 with a free $\Z_k$-action, and a structure map $\beta X \to B\Z_k$
 classifies the corresponding $k$-folded covering $\p X \to \beta X=
 \p X/\Z_k$. Here $N(\beta X)$ is given by inserting the cone
 $C\!\left<k\right>$ instead of $\left<k\right>$ in the fibre bundle
 $\p X \to \beta X$. Assuming that all manifolds are spin, we obtain
 corresponding bordism groups
 $\Omega^{\spin,(\left<k\right>\!,\Z_k)\fb}_*$ and the corresponding
 transformation $\alpha^{(\left<k\right>\!,\Z_k)\fb}\colon
 \Omega^{\spin,(\left<k\right>\!,\Z_k)\fb}_* \to KO_*(B\Z_k)$ which
 evaluates the index of the corresponding Dirac operator; see~\cite{BR} for details.
\item[{\rm (2)}] Let $\eta\in \Omega_1^{\spin}$ be as above, i.e., $[L]=\eta$ and $G=S^1$. Then, similarly, we arrive at the bordism
 groups $\Omega^{\spin,\eta\fb}_*$ and the index map
 $\alpha^{\eta\fb}\colon \Omega^{\spin,\eta\fb}_*\to KO^{\eta\fb}_*$, as
 in Theorem~\ref{Th1} as above; see~\cite{BPR,BR} for details.
\end{enumerate}
The above examples lead to the following two corollaries of Theorem~\ref{thmA}:
\begin{Corollary}
 Let $X_{\Sigma}$ be a spin $(\left<k\right>\!\fb)$-manifold. Assume $\dim
 X\geq 7$ and that $X$ and $\beta X$ are simply-connected. Then the homotopy
 type of the space $\Riem^{\psc}(X_{\Sigma})$ is a bordism invariant
 and depends only on the bordism class $[X_{\Sigma}]\in
 \Omega^{\spin,\left<k\right>\fb}_n$.
\end{Corollary}

\begin{Corollary}
 Let $X_{\Sigma}$ be a spin $(\eta\fb)$-manifold. Assume $\dim X\geq
 9$ and that $X$ and $\beta X$ are simply-connected. Then the homotopy type
 of the space $\Riem^{\psc}(X_{\Sigma})$ is a bordism invariant and
 depends only on the bordism class $[X_{\Sigma}]\in
 \Omega^{\spin,\eta\fb}_n$.
\end{Corollary}
\begin{Remark}
 To simplify the presentation, we consider only the spin case
 here. However, there are also non-spin cases where similar results
 hold. We~leave this intentionally outside of this paper.
\end{Remark}
The cases addressed in Theorems~\ref{thm:Lbdy} and~\ref{thm:BaasHP2}
give interesting implications.
\begin{enumerate}\itemsep=0pt
\item[{\rm (3)}] Let $L$ and $G$ be as in Theorem~\ref{thm:Lbdy}, i.e., $G$ is a simply connected Lie group, $L$ is a spin boundary, say
 $L=\partial\bar L$, with the standard metric $g_L$ on $L$ extending
 to a psc-metric on $\bar L$, and with the $G$-action on $L$
 extending to a $G$-action on $\bar L$. Then an $(L,G)$-singular spin
 manifold $X_\Sigma$ determines an element in the relevant bordism group
 $\Omega_n^{\spin,(L,G)-\fb}$ which fits to an exact triangle (see~\cite{BPR}):
\begin{gather*}
 \begin{diagram}
 \setlength{\dgARROWLENGTH}{1.95em}
 \node{\Omega^{\spin}_*}
 \arrow[2]{e,t}{i}
 \node[2]{\Omega^{\spin,(L,G)\fb}_*}
 \arrow{sw,t}{\beta}
 \\
 \node[2]{\Omega^{\spin}_{*}(BG).}
 \arrow{nw,t}{T} \node {}
 \end{diagram}
\end{gather*}
Here the indices $\alpha(\beta X)\in KO_{n-\ell-1}$ and $\alpha_{{\rm
 cyl}}(X)\in KO_n$ can be thought of as homomorphisms from
$\Omega^{\spin,(L,G)\fb}_*$ to a relevant $K$-theory.
\item[{\rm (4)}] Let $L=\mathbb{HP}^{2k}$ and $G={\rm Sp}(2k+1)$, $n\ge 1$.
 Assume that $\p X = \beta X \times L$, i.e., the $L$-bundle over
 $\beta X$ is trivial, or in other words that the singularities are
 of Baas--Sullivan type. Then a closed $(L,G)$-singular spin manifold
 $X_\Sigma$ determines an element in the corresponding bordism group
 $\Omega^{\spin,(L,G)\fb}_n$, and, as above, the index homomorphism
 from $\Omega^{\spin,(L,G)\fb}_*$ to a relevant $K$-theory.
\end{enumerate}
These examples lead to following corollary
\begin{Corollary}
 In both of the cases described in {\rm (3)} and {\rm (4)}, the homotopy type of the space
 $\Riem^{\psc}(X_{\Sigma})$ is a bordism invariant and depends on the
 bordism class $[X_{\Sigma}]\in \Omega^{\spin,(L,G)\fb}_n$, provided
 $n-\ell \geq 6$ $($where $\ell =2k$ in the case $(4))$.
\end{Corollary}
\begin{Remark}
To see how Theorem~\ref{thmA} implies the corollaries, it is enough to notice
that in all those cases, the dimensional assumptions imply that a
cobordism $W_{\Sigma}\colon X_{\Sigma}\bord X_{\Sigma,1}$ could be
modified so that the embedding $X_{\Sigma,1} \subset W_{\Sigma}$ is
two-connected in an appropriate sense, i.e., the induced embeddings
$\beta X_1\subset \beta W$ and $(X_1,\p X_1) \rh (W,\p W)$ are both
2-connected. Then the argument goes exactly as in the case of closed
manifolds.
\end{Remark}
In the last section we show that the cases (3) and (4) above lead to an
interesting result concerning homotopy groups of
$\Riem^{\psc}(X_{\Sigma})$.

\section{Preliminaries}\label{prim}
\subsection{Positive scalar curvature on manifolds with boundary}
Here we recall the main constructions and results from~\cite{W2}. The
set-up is as follows. Given a smooth compact $n$-dimensional manifold
$X$ (possibly with boundary $\p X\neq\varnothing$), we denote by~$\Riem(X)$, the space of all Riemannian metrics on $X$. The space
$\Riem(X)$ is equipped with the standard $C^{\infty}$-topology, giving
it the structure of a Fr{\'e}chet manifold; see~\cite[Chapter 1]{TW}
for details. For~each metric $g\in \Riem(X)$, we denote by
$s_{g}\colon X\rightarrow \mathbb{R}$ the scalar curvature on $X$ of the
metric $g$ and by $\Riem^{\psc}(X)\subset \Riem(X)$ the subspace of
psc-metrics on $X$.

In the case when $\p X\neq\varnothing$, it is
necessary to consider certain
subspaces of $\Riem^{\psc}(X)$, where metrics satisfy
particular boundary constraints.
With this in mind, we specify a collar embedding $c\colon\p
X\times[0,2)\hookrightarrow X$ around $\p X$ and
 define the space $\Riem^{\psc}(X, \p X)$ as:
\begin{gather*}
 \Riem^{\psc}(X, \p X):=\big\{h\in \Riem^{\psc}(X)\colon c^{*}h|_{\p
 X\times I}=h|_{\p X}+{\rm d}t^{2}\big\},
\end{gather*}
where $I:=[0,1]\subset [0,2)$.
Fixing a particular metric $g\in\Riem^{\psc}(\p X)$, we define the
subspace $\Riem^{\psc}(X, \p X)_g\subset \Riem^{\psc}(X, \p X)$ of all
psc-metrics $h\in\Riem^{\psc}(X, \p X)$, where $(c^{*}h)|_{\p X\times
 \{0\}}=g$.

Let $Z\colon Y_0\bord Y_1$ be a bordism between $(n-1)$-dimensional
manifolds $Y_0$ and $Y_1$ given together with collars $c_i\colon Y_i \times
[0,2) \rh Z$, $i=0,1$ near the boundary $\p Z = Y_0\sqcup Y_1$. Then
 $\Riem^{\psc}(Z, \p Z)$ denotes the space of psc-metrics on $Z$ which
 restrict as a product structure on the neighbourhood
 $c_i(Y_{i}\times I)\subset Z$, $i=0,1$;
 i.e., $c_i^{*}\bar{g}=g_i+{\rm d}t^{2}$ on $Y_i\times I$ for some pair of
 metrics $g_i\in \Riem^{\psc}(Y_i)$, $i=0,1$.
Now we fix a pair of psc-metrics $g_0\in\Riem^{\psc}(Y_0)$ and
$g_1\in\Riem^{\psc}(Y_1)$ and consider the following subspace of
$\Riem^{\psc}(Z, \p Z)$:
\begin{gather*}
\Riem^{\psc}(Z, \p Z)_{g_0, g_1}:=\big\{ \bar{g}\in\Riem^{\psc}(Z, \p Z) \mid
c_i^{*}\bar{g}=g_i+{\rm d}t^{2} \text{ on } Y_i\times [0,1], \, i=0,1\big\}.
\end{gather*}
We note that each metric $\bar g\in \Riem^{\psc}(Z, \p Z)_{g_0, g_1}$ provides a
psc-bordism $(Z,\bar g)\colon (Y_0,g_0)\bord (Y_1,g_1)$. We~next assume $X$ is a manifold whose boundary $\p X = Y_0$ is equipped
with the metric $g_0$. Furthermore, we assume that both spaces
$\Riem^{\psc}(X, \p X)_{g_0}$ and $\Riem^{\psc}(Z, \p Z)_{g_0, g_1}$ are
non-empty. Now, by making use of the relevant collars, we glue together
$X$ and $Z$ to obtain a smooth manifold which we denote $X\cup Z$ and
which has boundary $\p (X\cup Z)=Y_1$; see Figure~\ref{bordism-a}.

In particular, we obtain the space $\Riem^{\psc}(X\cup Z, Y_1)_{g_1}$ of
psc-metrics which restrict as ${g_1}+{\rm d}t^{2}$ on $c_1(Y_1\times [0,1])
\subset Z \subset X\cup Z$. Then for any metric
$\bar{g}\in\Riem^{\psc}(Z, \p Z)_{g_0, g_1}$, we obtain a map:
\begin{align}
 \mu_{Z,\bar{g}}\colon\ \Riem^{\psc}(X, \p X)_{g_0}&
 \longrightarrow \Riem^{\psc}(X\cup Z, Y_1)_{g_1},\nonumber
 \\
 h&\longmapsto h\cup \bar{g},\label{gluemap-1}
\end{align}
where $h\cup \bar{g}$ is the metric obtained on $X\cup Z$ by the
obvious gluing depicted in Figure~\ref{bordism-a}.

\begin{figure}[!htbp]
\hspace{25mm}
\begin{picture}(0,0)
\includegraphics{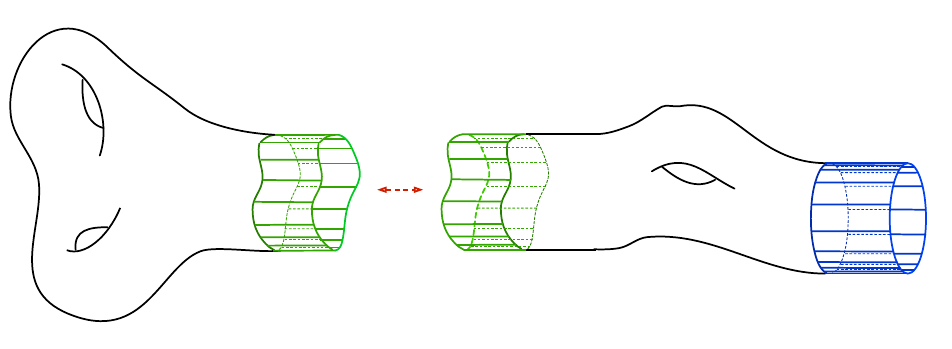}%
\end{picture}
\setlength{\unitlength}{3947sp}%
\begingroup\makeatletter\ifx\SetFigFont\undefined%
\gdef\SetFigFont#1#2#3#4#5{%
 \reset@font\fontsize{#1}{#2pt}%
 \fontfamily{#3}\fontseries{#4}\fontshape{#5}%
 \selectfont}%
\fi\endgroup%
\begin{picture}(5079,1559)(1902,-7227)
\put(1400,-6000){\makebox(0,0)[lb]{\smash{{\SetFigFont{10}{8}{\rmdefault}{\mddefault}{\updefault}{\color[rgb]{0,0,0}$(X, h)$}%
}}}}
\put(3050,-7000){\makebox(0,0)[lb]{\smash{{\SetFigFont{10}{8}{\rmdefault}{\mddefault}{\updefault}{\color[rgb]{0,0,0}$g_0+{\rm d}t^{2}$}%
}}}}
\put(4000,-7000){\makebox(0,0)[lb]{\smash{{\SetFigFont{10}{8}{\rmdefault}{\mddefault}{\updefault}{\color[rgb]{0,0,0}$g_0+{\rm d}t^{2}$}%
}}}}
\put(5700,-7130){\makebox(0,0)[lb]{\smash{{\SetFigFont{10}{8}{\rmdefault}{\mddefault}{\updefault}{\color[rgb]{0,0,0}$g_1+{\rm d}t^{2}$}%
}}}}
\put(5000,-5900){\makebox(0,0)[lb]{\smash{{\SetFigFont{10}{8}{\rmdefault}{\mddefault}{\updefault}{\color[rgb]{0,0,0}$(Z, \bar{g})$}%
}}}}
\end{picture}%
\caption{Attaching $(X, h)$ to $(Z,\bar{g})$ along a common boundary $\p X=Y_0$.}
\label{bordism-a}
\end{figure}
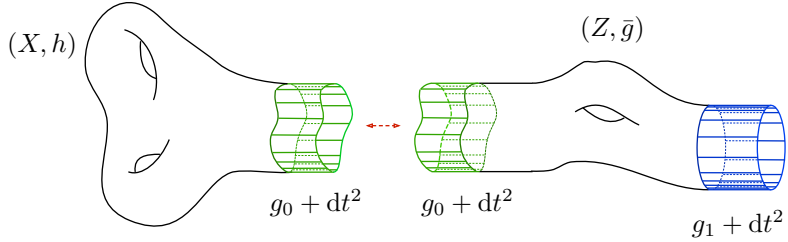
Consider the case when the bordism $Z\colon Y_0\bord Y_1$ is an
\emph{elementary bordism}, i.e., when $Z$ is the trace of a surgery on
$Y_0$ with respect to an embedding $\phi\colon S^p\times D^{q+1}\to Y_0$
with $p+q+1=n-1=\dim Y_0$. Then we have the following.

\begin{Lemma}[surgery lemma; see~\cite{Ch,EF,W1}]
 Let $g_0\in \Riem^+(Y_0)$ be any metric. Assume $q\geq 2$. Then
 there exist metrics $g_1\in \Riem^+(Y_1)$ and $\bar g\in
 \Riem^+(Z,\p Z)_{g_0,g_1}$ such that $(Z,\bar g)\colon (Y_0,g_0)\bord
 (Y_1,g_1)$ is a psc-bordism.
\end{Lemma}
Such a bordism is usually called a \emph{Gromov--Lawson bordism} (or
\emph{GL-bordism} for short). Here is a reformulation of the main
technical result from~\cite{W2}:
\begin{Theorem}\label{thm-boundary}
Let $Z\colon Y_0\bord Y_1$ be an elementary bordism as above with $p,q\geq
2$. Then for any metric $g_0\in \Riem^+(Y_0)$ there exist metrics
$g_1\in \Riem^+(Y_1)$ and $\bar g\in \Riem^+(Z,\p Z)_{g_0,g_1}$ such
that the map
\begin{gather*}
\mu_{Z,\bar{g}}\colon\ \Riem^{\psc}(X, \p X)_{g_0} \xrightarrow{\simeq} \Riem^{\psc}(X\cup Z, Y_1)_{g_1}
\end{gather*}
defined by {\rm (\ref{gluemap-1})}, is a weak homotopy equivalence.
\end{Theorem}

\subsection{Lifting to a Riemannian submersion}\label{submersionsection}
An important aspect of our work involves lifting psc-metrics
from the base of a smooth fibre bundle to a psc-submersion metric
on its total space. In this section we briefly recall some pertinent facts and
establish some notation.

Let $p\colon M\rightarrow B$, be the associated smooth fibre bundle to a
principal $G$-bundle with respect to a Lie group, $G$, a compact
smooth base manifold $B$ and compact smooth fibre manifold $F$, on
which $G$ acts. Thus, the bundle has a structure map $f\colon B\rightarrow
BG$. We~denote by $\mathcal{HD}(p)$ the space of all horizontal
distributions. This space is topologised in the obvious way as a
subspace of smooth sections of the Grassmann bundle obtained by
replacing each tangent space to $M$ with the Grassmannian of $\dim
B$-dimensional subspaces. The following fact is elementary.
\begin{Proposition}
The space $\mathcal{HD}(p)$ is convex.
\end{Proposition}
Suppose we have a $G$-invariant metric on $F$, denoted $g_F$, a metric
on $B$ denoted $g_B$ and a~horizontal distribution
$\mathcal{H}\in\mathcal{HD}(p)$. By a well-known construction
\cite[Proposition~9.59]{Besse}, the triple, $(g_{B}, g_{F}, \mathcal{H})$,
determines a unique Riemannian submersion metric (with totally
geodesic fibre metrics isometric to~$g_F$) on the total space~$M$. For~simplicity, we fix the $G$-invariant metric on $F$, $g_F$, and
consider constant multiples of this metric, $\tau g_{F}$, for some
$\tau\in(0,\infty)$. This ability to scale up or down the fibre metric
will give us important flexibility later on. We~consider only the case
when the fibre metric $g_{F}$ has non-negative scalar curvature and
the base metric on $B$ has positive scalar curvature. Thus, we assume
$\Riem^{\psc}(B)\neq \varnothing$.

We notice that the construction of a submersion metric varies
continuously with respect to the triple $(g_{B}, \tau g_{F},
\mathcal{H})$ and so we obtain a map
\begin{gather*}
\mathcal{S}\colon\ \Riem^{\psc}(B)\times \mathcal{HD}(p)\times (0,\infty)\longrightarrow \Riem (M),
\end{gather*}
where $\mathcal{S}(h, \mathcal{H}, \tau)$ is the unique submersion
metric with base $(B, h)$, fibre $(F, \tau g_{F})$ and horizontal
distribution $\mathcal{H}$.

We denote by $\Riem_{\mathcal{S}} (M)$, the image of the map,
$\mathcal{S}$, in $\Riem (M)$ and by $\Riem_{\mathcal{S}}^{\psc} (M)$
its subspace of~psc-metrics. Thus, each element of
$\Riem_{\mathcal{S}}^{\psc} (M)$ is a psc-submersion metric with
respect to some base metric $g_B\in\Riem^{\psc}(B)$, some fibre metric
$\tau g_{F}$ on $F$ (for some scaling constant $\tau>0$) and some
horizontal distribution $\mathcal{H}\in\mathcal{HD}(p)$. There are
obvious maps from the space $\Riem_{\mathcal{S}}^{\psc} (M)$ to the
spaces $\mathcal{HD}(p)$, $(0,\infty)$ and $\Riem^{\psc}(B)$, sending
a submersion metric to its respective horizontal distribution, fibre
scaling constant and base metric. We~consider the latter, denoting it
\begin{gather*}
\mathcal{B}\colon\ \Riem_{\mathcal{S}}^{\psc} (M)\rightarrow \Riem^{\psc}(B).
\end{gather*}
Thus, for any psc-submersion
$\bar{g}\in\Riem_{\mathcal{S}}^{\psc}$, $\mathcal{B}(\bar{g})$ is the
base metric on $B$. Finally, for any base metric $h\in \Riem^{\psc}
(B)$, we denote by $\Riem_{\mathcal{S}}^{\psc} (M)_{h}$, the pre-image
$\mathcal{B}^{-1}(h)$.

We will be interested in constructing a ``lifting" map
\begin{gather*}
\Riem^{\psc}(B)\to \Riem^{\psc}(M)
\end{gather*}
which sends each psc-metric on $B$, to a submersion metric on the
the total space of the bundle $p\colon M\to B$. This will require
some preliminary work.

From~\cite[Proposition 9.70]{Besse} we know that the scalar curvature
of such a metric, $\mathcal{S}(h,\mathcal{H},\tau)$, denoted $s_M$, is
given by the following formula of O'Neill:
\begin{gather*}
s_M=s_{h}\circ p +\frac{1}{\tau}s_{F}-\tau|A|^{2},
\end{gather*}
where $s_h$ and
$s_F$ are the scalar curvatures of the metrics $h\in\Riem^{\psc}(B)$
and $g_F$, and $A$ is the well-known tensor which is the obstruction
to the integrability of the distribution. Note that the other
well-known terms in this formula, in particular the tensor, $T$,
vanish since the fibres are totally geodesic~\cite[Theorem~9.59]{Besse}. By assumption, the sum of the first two terms in this
formula is positive and so in order to ensure that $s>0$ we need to
minimise the effect of $A$.

This formula varies continuously with respect to the choices
of $h$ and $\mathcal{H}$; recall $g_F$ is fixed. In particular, we write
$A:=A_{h, \mathcal{H}}$.
Thus, by the compactness of $M$,
there are continuous real valued parameters
\begin{gather*}
 \mathrm{m}(h)=\min\{s_{h}\circ p(x)\colon x\in M\}\qquad
 \text{and} \qquad
 \mathrm{M}_{A}(h, \mathcal{H})=\max\{|A_{h, \mathcal{H}}(x)|\colon x\in M\}.
\end{gather*}
It follows that the scalar curvature, $s$, of the metric
$\mathcal{S}(h,\mathcal{H},\tau)$ satisfies
\begin{gather*}
s\geq \mathrm{m}(h) +\frac{1}{\tau}s_{L}-\tau \mathrm{M}_{A}(h,
\mathcal{H})^{2}.
\end{gather*}
To ensure positivity of $s$, we define $\bar{\tau}(h, \mathcal{H})$ by
\begin{gather}\label{taumin}
 \bar{\tau}(h, \mathcal{H}):=\frac{\mathrm{m}(h)}{2 \mathrm{M}_{A}(h, \mathcal{H})^{2}}.
\end{gather}
Thus, we obtain a map
\begin{align*}
 \mathcal{S}^{\psc}\colon\ \Riem^{\psc}(B)\times\mathcal{HD}(p)\to &\
 \Riem_{\mathcal{S}}^{\psc} (M),\nonumber
 \\
(h,\mathcal{H})\mapsto&\
\mathcal{S}(h,\mathcal{H},\bar{\tau}(h, \mathcal{H})).
\end{align*}
By fixing some horizontal distribution, $\mathcal{H}$, we define
the map
\begin{gather*}
 \mathcal{S}_{\mathcal{H}}^{\psc}\colon\ \Riem^{\psc}(B)\rightarrow
 \Riem_{\mathcal{S}}^{\psc} (M),
\end{gather*}
by $\mathcal{S}_{\mathcal{H}}^{\psc}(h)=\mathcal{S}^{\psc}(h, \mathcal{H})$.

We now deduce some facts about this construction concerning homotopy. As~an immediate consequence of the convexity of $\mathcal{HD}(p)$, we
obtain the following.
\begin{Proposition}
For any horizontal distributions, $\mathcal{H}, \mathcal{H}'\in
\mathcal{HD}(p)$, the maps $ \mathcal{S}_{\mathcal{H}}^{\psc}$ and $
\mathcal{S}_{\mathcal{H}'}^{\psc}$ are homotopy equivalent.
\end{Proposition}

\noindent
The lemma below will play an important role later on.
\begin{Lemma}\label{lifthe}
The maps $\mathcal{B}$ and $\mathcal{S}_{\mathcal{H}}^{\psc}$ are homotopy inverse, thus determining a homotopy equi\-valence:
\begin{gather*}
\Riem^{\psc}(B)\simeq\Riem_{\mathcal{S}}^{\psc} (M).
\end{gather*}
\end{Lemma}

\begin{proof}
The composition $\mathcal{B}\circ\mathcal{S}_{\mathcal{H}}^{\psc}$ is
precisely the identity map on $\Riem^{\psc}(B)$. Consider now an
arbitrary element $\bar{g}\in\Riem_{\mathcal{S}}^{\psc} (M)$. This
element is uniquely represented as a triple $\bar{g}=(\bar{g}_{B},
\mathcal{H}_{\bar{g}}, \tau_{\bar{g}})$, consisting of a base metric
$\bar{g}_{B}=\mathcal{B}(\bar{g})$, a horizontal distribution and a
scaling constant. Now,
\begin{gather*}
\mathcal{S}_{\mathcal{H}}^{\psc}\circ
\mathcal{B}(\bar{g})=\mathcal{S}(\bar{g}_{B}, \mathcal{H},
\bar{\tau}(\bar{g}_{B}, \mathcal{H})).
\end{gather*}
Thus, the desired homotopy is
given by
\begin{gather*}
H_{t}(\bar{g}):=\mathcal{S}(\bar{g}_{B},
(1-t)\mathcal{H}+t\mathcal{H}_{\bar{g}}, \bar{\tau}(\bar{g}_{B},
(1-t)\mathcal{H}+t\mathcal{H}_{\bar{g}})),
\end{gather*}
where $t\in [0,1]$. That the resulting submersion metric
$H_{t}(\bar{g})$ has positive scalar curvature for all $t$ is
guaranteed by the construction of the function $\bar{\tau}$.
\end{proof}
We now consider a variation of the scenario above which will be of
great importance for us. We~replace the bundle $p\colon M\rightarrow B$,
with a smooth fibre bundle $\bar{p}\colon\bar{E}\to \bar{B}$. In this case,
the base $\bar{B}$ and total space are smooth compact bordisms of
manifolds, denoted $\bar B\colon B_0 \bord B_1$ and $\bar E\colon E_0\bord
E_1$, respectively. Here $E_i=E|_{B_i}$ for $i\in\{0,1\}$. This bundle
is also a $G$-bundle with fibre $L$, as above, and with structure map
$\bar{f}\colon\bar{B}\to BG$ restricting to structure maps $f_{i}\colon B_{i}\to
BG$ of the corresponding boundary bundles
$p_{i}=\bar{p}|_{E_i}\colon E_{i}\to B_i$. These boundary bundles are of
exactly the type we discussed earlier with the role of $M$ played by
$E_i$. We~equip the bordism $\bar B$ with collars $c_i\colon B_i\times
[0,2)\to \bar{B}$, $i=0,1$, along the boundary $\p \bar B$. We~will
 now introduce the following simplifying notation:
\begin{gather*}
\bar{B}_{c_i}:=c_{i}(B_0\times [0,1]),\qquad
\bar{B}_{c_i}':=\bar{B}\setminus \bar{B}_{c_i},
\\
\bar{E}_{c_i}:=\bar{E}|_{\bar{B}_{c_i}}\qquad \text{and}\qquad
\bar{E}_{c_i}':=\bar{E}\setminus \bar{E}_{c_i}, \qquad \text{where}\quad
i\in\{0,1\}.
\end{gather*}
As usual, we assume that $L$ is equipped with a fixed metric $g_L$ of
non-negative scalar curvature. Suppose that $h_{0}$ and $h_1$ are a
pair of psc-metrics on $B_0$ and $B_1$ and that $\bar{h}$ is a
psc-metric in~the (assumed to be non-empty) space of psc-bordism
metrics $\Riem^{\psc}\big(\bar{B}, \p \bar{B}\big)_{h_0, h_1}$. Given some
horizontal distribution $\bar{\mathcal{H}}$ on the bundle $\bar{p}$,
the construction above allows us to ``lift" the psc-metric $\bar{h}$
to psc-submersion metrics on the total space $\bar{E}$. In particular,
to ensure our lifted metric has a product structure near the boundary
of $\bar{E}$, we can impose the following admissibility condition on
horizontal distributions.
\begin{Definition}
A horizontal distribution, $\bar{\mathcal{H}}$ on the bundle
$\bar{p}\colon \bar{E}\to \bar{B}$ is said to be {\em admissible} provided
it takes the form of a product $\mathcal{H}_{i}\times\mathbb{R}$ on
the restrictions $\bar{E}_{c_i}\to \bar{B}_{c_i}$ for each
$i\in\{0,1\}$.
\end{Definition}
\noindent We denote by $\mathcal{HD}_{\mathrm{adm}}(\bar{p})$, the
subspace of admissible horizontal distributions on $\bar{p}$. That
this space is convex follows from an elementary exercise.

For our purpose, we need to do something slightly more delicate. Let
us assume that the boundary component $E_0$ of the total space is
already equipped with a psc-metric $g_0$, one deter\-mi\-ning a Riemannian
submersion of positive scalar curvature, $p_0\colon (E_0, g_0)\to (B_0,
h_0)$, with respect to the fibre $(L, \tau_0 g_L)$ for some constant
$\tau_0>0$, and horizontal distribution
$\mathcal{H}_0\in\mathcal{HD}(p_0)$. Thus,
$g_0\in\Riem_{\mathcal{S}}^{\psc}(E_0)$ and more specifically lies in
the image of the map
\begin{gather*}
S_{\mathcal{H}_0}^{\psc}\colon \ \Riem^{\psc}(B_0)\to
\Riem_{\mathcal{S}}^{\psc}(E_0).
\end{gather*}
We wish to extend the psc-metric
$g_0$ to a certain psc-metric on the total space $\bar{E}$. This is
done in the following lemma.
\begin{Lemma}\label{pscliftlemma} Let $\bar{p}\colon \bar{E}\to \bar{B}$
 denote the bundle described above with structure map
 $\bar{f}\colon\bar{B}\to BG$ and fibre $L$, equipped with a fixed metric
 $g_L$ of non-negative scalar curvature. Let
 \begin{gather*}
 p_0=\bar{p}|_{E_0}\colon\ (E_0, g_0)\to (B_0, h_0)
 \end{gather*}
 be a psc-Riemannian
 submersion with respect to the fibre $(L, \tau_0 g_L)$ for some
 constant $\tau_0>0$, and horizontal distribution
 $\mathcal{H}_0\in\mathcal{HD}(p_0)$. For~compact spaces $K_1$ and
 $K_2$, let $\bar{h}(x)\in \Riem^{\psc}\big(\bar{B}, \p \bar{B}\big)$, be~a~continuous family of psc-bordism metrics on $\bar{B}$, indexed by
 $x\in K_1$, and let $\bar{\mathcal{H}}(y)\in
 \mathcal{HD}_{\mathrm{adm}}(\bar{p})$ be a continuous family of
 admissible horizontal distributions indexed by $y\in K_2$. We~further assume that, for all $x\in K_1$ and $y\in K_2$,
 $\bar{h}(x)|_{B_0}=h_0$ and $\bar{\mathcal{H}}(y)|_{E_0}=\mathcal{H}_0$.

Then there is a constant $\bar{\tau}_{\min}>0$ and a corresponding
family of psc-metrics $\bar{g}(x,y)$, $(x,y)\in K_1\times K_2$ on
$\bar{E}$ which satisfies the following conditions.
\begin{enumerate}\itemsep=0pt
\item[$(i)$] $\bar{p}\colon \big(\bar{E}, \bar{g}(x,y)\big)\to
 \big(\bar{B},\bar{h}(x)\big)$ is a Riemannian submersion with respect to the
 horizontal distribution $\bar{\mathcal{H}}(y)$.
\item[$(ii)$] Each psc-metric $\bar{g}(x,y)$ takes the form of a
 product ${g}_{i}(x,y)+{\rm d}t^{2}$ near the boundary components $E_i$,
 and satisfies ${g}_{0}(x,y)=g_0$ on $E_0$.
\item[$(iii)$] The Riemannian submersion obtained by restriction
 to $E_1$,
 \begin{gather*}
 p_1=\bar{p}|_{E_1}\colon\ \big(E_1,
 g_1(x,y):=\bar{g}(x,y)|_{E_1}\big)\to\big(B_1, h_1(x):=\bar{h}(x)|_{B_1}\big),
 \end{gather*}
 has fibre $(L, \bar{\tau}_{\min} g_L)$.
\end{enumerate}
\end{Lemma}

\begin{proof}
We begin by specifying $\bar{g}=\bar{g}(x,y)$ as, for all $(x,y)$, the
product $g_0+{\rm d}t^2$ near the boundary component $E_0$. More precisely,
$\bar g:= g_0+{\rm d}t^2$ on $\bar{E}_{c_0}\cong E_{0}\times [0,1]$. Thus,
near $E_0$, the metric $\bar g$ forms part of a submersion $p_0\times
\mathrm{Id}\colon \big(E_0\times I, g_0+{\rm d}t^2\big)\rightarrow \big(B_0\times I,
h_0+{\rm d}t^{2}\big)$ with fibre $(L,\tau_0 g_L)$ and horizontal distribution
$\bar{\mathcal{H}}|_{\bar{E}_{c_0}}=\mathcal{H}_0\times \mathbb{R}$.

Our next task is to extend this metric as a psc-submersion over the
rest of $\bar{E}$. To maintain positive scalar curvature it may be
necessary to continuously rescale the fibre metric, but in such a way
as to preserve the original fibre metric near $E_0$. To simplify
things, we make adjustments to the fibre metric only in
$\bar{E}_{c_0}$ (and away from $E_{0}$), where the metric, $\bar g$,
takes the form of a~product $\bar g=g_0+{\rm d}t^{2}$. Consider, for
constants $\tau>0$ and $b_{\tau}\geq 2$, smooth monotonic functions
$\gamma_{\tau}\colon [0, b_{\tau}]\rightarrow [0, \tau]$ satisfying
\begin{enumerate}\itemsep=0pt
\item[$(i)$] $\gamma_{\tau}(t)=\tau_0$ when $t\in[0,1]$,
\item[$(ii)$] $\gamma_{\tau}(t)=\tau$ when $t\in[b_{\tau}-1,b_{\tau}]$.
\end{enumerate}
We define $g_0(t)$ to be the submersion metric obtained on
$E_0\times\{t\}\to B_0\times\{t\}$ with base metric~$h_0$ and
horizontal distribution $\mathcal{H}_0$ as before but replacing the
fibre metric $\tau_0 g_{L}$ with~$\gamma_{\tau}(t)g_{L}$. Now replace
the submersion $p_0\times \Id\colon \big(\bar{E}_{c_0}=E_0\times I,
g_0+{\rm d}t^{2}\big)\rightarrow \big(B_0\times I, h_0+{\rm d}t^{2}\big)$, with $p_0\times
\Id\colon \big(E_{0}\times[0,b_\tau], g({\gamma_{\tau}(t)})+{\rm d}t^{2}\big)\rightarrow
\big(B_0\times [0,b_\tau], h_0+{\rm d}t^{2} \big)$. The following proposition is an~easy consequence of a well-known rescaling technique; see for example
\cite[Lemma~3]{GL} or~\cite[Lemma~1.3]{W0}.
\begin{Proposition}
For any constant $\tau>0$, the smooth function $\gamma_{\tau}$ above
can be chosen so that the metric $g({\gamma_{\tau}(t)})+{\rm d}t^{2}$ has
positive scalar curvature.
\end{Proposition}
\noindent We will specify the choice of $\tau>0$ shortly. In the mean
time, let us assume that $\gamma_{\tau}$ has been chosen to guarantee
positive scalar curvature. Note that near $t=0$, the metric remains
unchanged.

Consider now the other end of the cylinder, where $t=\tau_b$. Here we
have a product of~psc-submersion metrics differing only from that at
the $t=0$ end in that the fibre metric is now~$\tau g_{L}$. No change
has been made to the base metric and it is assumed to extend over the
base bordism~$\bar{B}$ as a psc-metric $\bar{h}$ which takes the form
of a product metric ${h}_{1}+{\rm d}t^{2}$ near the boundary component~$B_{1}$. For~each $y$, the admissible horizontal distribution
$\bar{\mathcal{H}}(y)$ extends $\mathcal{H}_0\times \mathbb{R}$ over
the rest of $\bar{E}$, inducing, outside of the collar
$\bar{E}_{c_0}$, Riemannian submersion metrics with base metric
$\bar{h}|_{\bar{B}_{c_0}'}(x)$ and fibre metric $\tau g_{L}$. We~need
to ensure positive scalar curvature by choosing appropriately small
$\tau$.

Suppose for any $(x,y)$, we choose $\tau=\bar{\tau}\big(\bar{h}(x),
\bar{\mathcal{H}}(y)\big)>0$, as defined in (\ref{taumin}). This ensures,
via the canonical variation formula \cite[Proposition~9.70]{Besse}, that
the submersion metric
\begin{gather*}
\mathcal{S}\big(\bar{h}(x), \bar{\mathcal{H}}(y),
\bar{\tau}\big(\bar{h}(x), \bar{\mathcal{H}}(y)\big)\big)
\end{gather*}
on the bundle
$\bar{E}_{c_0}'\to \bar{B}_{c_{0}}'$ with base metric
$\bar{h}(x)|_{\bar{B}_{c_{0}'}}$, fibre metric $\bar{\tau}\big(\bar{h}(x),
\bar{\mathcal{H}}(y)\big) g_{L}$ and horizontal distribution
$\bar{\mathcal{H}}(y)|_{\bar{E}_{c_0}'}$ has positive scalar
curvature. Thus, we set
\begin{gather*}
\tau=\bar{\tau}_{\min}:=\min\big\{\bar{\tau}\big(\bar{h}(x),
\bar{\mathcal{H}}(y)\big)\colon (x,y)\in K_1\times K_2\big\},
\end{gather*}
to obtain a single
scaling factor which works for all $(x,y)$. Finally, the metric
$\bar{g}(x,y)$ is defined to be the submersion metric determined by
the base $\big(\bar{B}, \bar{h}(x)\big)$, fibre $L$ with varying metric~$\gamma_{\bar{\tau}_{\min}}(t) g_L$ on $E_{c_0}$ and
$\bar{\tau}_{\min} g_{L}$ on the rest of~$\bar{E}$, as well as
horizontal distribution $\bar{\mathcal{H}}(y)$.
\end{proof}

\subsection{The case of manifolds with fibred singularities}
We now return to the manifolds with fibred singularities introduced in
Section~\ref{Intro}. Recall that~$X$ is assumed to be a smooth
manifold with boundary $\p X\neq \varnothing$, which is the total space
of a smooth bundle $\p X \to \beta X$ with the fibre $L$ and structure
map $f\colon \beta X\to BG$. Here $G$ is a subgroup of the isometry group of
$g_{L}$, a Riemannian metric on $L$. Usually the Riemannian manifold
$(L, g_L)$ is referred to as the \emph{link}.

As explained in Section~\ref{Intro}, we obtain a { manifold with
 fibred singularities}, $X_{\Sigma}$, given as $X_{\Sigma}:=X\cup_{\p
 X} -N(\beta X)$, by replacing the fibre $L$ with it's cone, $C(L)$,
to obtain bundle $N(\beta X) \to \beta X$. Such a manifold a
manifold is referred to as an {\em $(L,G)$-manifold.} For our
constructions, we will require some further constraints on the link
$(L,g_{L})$, characterised by the following definition.
\begin{Definition}
 A link $(L,g_L)$ is \emph{simple} if it satisfies one of the
 following conditions:
\begin{enumerate}\itemsep=0pt
\item[$(a)$] $(L,g_L)$ is homogeneous with scalar curvature, $s_{g_L}$,
 a positive constant,
\item[$(b)$] $(L,g_L)=\big(S^1,{\rm d}\theta^2\big)$,
\item[$(c)$] $L =\mathbb{Z}_k$.
\end{enumerate}
\end{Definition}
Henceforth, we assume that $(L, g_{L})$ is a simple link. Before
going forward with our constructions, we should clarify why the
condition that the scalar curvature $s_{g_L}$ is non-negative constant
is important here.

In the case $(a)$, $s_{g_L}$ is a positive constant. We~define the cone
metric $g_{C(L)}$ on the cone,~$C(L)$, by
\begin{gather}\label{conemetricdef}
g_{C(L)}={\rm d}t^{2}+c_{L}^{-2}t^{2}g_{L},
\end{gather}
where $c_L=\sqrt{\frac{\ell(\ell-1)}{s_{g_L}}}$ and $\dim
L=\ell$. This is a warped product metric on $(0,1]\times L$ away from
the cone point (where $t=0$) and is scalar flat, as demonstrated in
the appendix.

The case $(b)$, when $L=S^1$, has special features. As~$\dim L=\ell=1$
here, the definition of the metric $g_{C(L)}$ above coincides with the
more general cone metric construction detailed in~the appendix for the
case when $g_{L}$ is scalar flat. In particular, formula~\eqref{scalarflatcone} in the appendix gives that the scalar curvature
of the metric $g_{C(S^1)}$ is identically zero. Moreover, any smooth
$S^1$-bundle $p\colon Y \to B$ admits a free $S^1$-action on the manifold
$Y$ such that $B=Y/S^1$. Then, according to a result by
B\'erard-Bergery~\cite[Theorem~C]{B-B}, a manifold $Y$ admits an
$S^1$-equivariant psc-metric if and only if the orbit space, the
manifold $B=Y/S^1$ admits a psc-metric.

In the case $(c)$ $L=\mathbb{Z}_k$. Then the cone $C(\mathbb{Z}_k)$ has
the standard Euclidian metric, and we do not make any further
assumptions. In any case, we can assume that the cone metric
$g_{C(L)}$ is always scalar-flat (outside of its vertex).

We return now to the manifold $X$, which forms part a pseudo-manifold
with $(L,G)$-sin\-gu\-la\-ri\-ties, $X_{\Sigma}=X\cup_{\p X} -N(\beta X)$.
Recall that $\p X$ forms part of $G$-bundle, $p\colon \p X\to \beta X$ with
$G$ a~subgroup of the isometry group of a fixed metric $g_L$ on the
fibre $L$. We~denote by $f\colon \beta X \to BG$, a structure map for this
bundle. Throughout, $\dim X = n$, $\dim L= \ell$ and we reiterate
that the link $(L, g_{L})$ is assumed to be simple.

Bearing in mind the notions described in Section~\ref{submersionsection}, we have a map
\begin{gather}\label{Slift}
\mathcal{S}\colon \ \Riem^{\psc}(\beta X)\times \mathcal{HD}(p)\times
(0,\infty)\longrightarrow \Riem (\p X)
\end{gather}
which sends any triple $(h, \mathcal{H}, \tau)$ to the unique
submersion metric on $\p X$ with base $(\beta X, h)$, fibre $(L, \tau
g_L)$ and horizontal distribution $\mathcal{H}$. Recall that the image
of $\mathcal{S}$ is denoted $\Riem_{\mathcal{S}}(\p X)$ and we further
denote by $\Riem_{\mathcal{S}}^{\psc}(\p X)$ its subspace of
psc-metrics. Thus, elements of $\Riem_{\mathcal{S}}^{\psc}(\p X)$ are
psc-submersion metrics on the total space $\p X$ with respect to some
base metric in $\Riem^{\psc}(\beta X)$ and fibre metric $\tau g_L$ for
some $\tau>0$. Recall that, for any $h\in \Riem^{\psc}(\beta X)$, the
space $\Riem_{\mathcal{S}}^{\psc}(\p X)_h$ consists of all
psc-submersion metrics in $\Riem_{\mathcal{S}}^{\psc}(\p X)$ with base
metric $h$. Recall also that by specifying a parameter $\bar{\tau}(h,
\mathcal{H})$, we obtained ``lifting" maps
\begin{align*}
 \mathcal{S}^{\psc}\colon\ \Riem^{\psc}(\beta X)\times\mathcal{HD}(p)\to &\
 \Riem_{\mathcal{S}}^{\psc} (\p X),\nonumber
 \\
(h,\mathcal{H})\mapsto&\
\mathcal{S}(h,\mathcal{H},\bar{\tau}(h, \mathcal{H})),
\end{align*}
and
\begin{gather}\label{SliftplusH}
 \mathcal{S}_{\mathcal{H}}^{\psc}\colon\ \Riem^{\psc}(\beta X)\rightarrow
 \Riem_{\mathcal{S}}^{\psc} (\p X),
\end{gather}
defined $\mathcal{S}_{\mathcal{H}}^{\psc}(h):=\mathcal{S}^{\psc}(h,
\mathcal{H})$, for some fixed horizontal distribution $\mathcal{H}$.
In particular, we recall that the homotopy type of the map
$\mathcal{S}_{\mathcal{H}}^{\psc}$ is independent of the choice of
$\mathcal{H}$ and, moreover, the map
$\mathcal{S}_{\mathcal{H}}^{\psc}$ forms one direction in a homotopy
equivalence (Lemma~\ref{lifthe} above); its homotopy inverse is the
map, $\mathcal{B}\colon\Riem_{\mathcal{S}}^{\psc}(\p X)\to
\Riem^{\psc}(\beta X)$, sending submersions to their base metrics.
\begin{Remark}
We emphasise that the fibre metric $g_{L}$ and structure map $f\colon \beta
X\to BG$ are assumed to be fixed throughout. Thus, while the map
$\mathcal{S}$ (along with $\mathcal{S}^{\psc}$ and
$\mathcal{S}_{\mathcal{H}}^{\psc}$) depend of~course on these objects,
and could be denoted by something like $\mathcal{S}_{f, g_{L}}$, we
will not feature them in the notation.
\end{Remark}

\subsection{Defining well-adapted metrics}\label{welladaptsec}
Returning to the manifold $X$, we consider the space of all
psc-metrics, $\Riem^{\psc}(X, \p X)$, which take a product structure
near the boundary. Consider now the restriction map
\begin{gather*}
 \mathrm{res}\colon\ \Riem^{\psc}(X, \p X) \to \Riem^{\psc}(\p X),\qquad
\mathrm{res}\colon\ g \mapsto g|_{\p X}.
\end{gather*}
This map is very important for us because of the following fact:

\begin{Theorem}[\cite{Ch,EF}]
 The restriction map $\mathrm{res}\colon \Riem^{\psc}(X, \p X) \to
 \Riem^{\psc}(\p X)$ is a Serre fibre bundle.
\end{Theorem}
We will make significant use of this fact later on. For~now however,
we consider the pre-image of the space $\Riem_{\mathcal{S}}^{\psc} (\p
X)$.
\begin{Definition}
The space
\begin{gather*}
\Riem_{\mathcal{S}}^{\psc}(X, \p
X):=\mathrm{res}^{-1}\big(\Riem_{\mathcal{S}}^{\psc} (\p X)\big)\subset
\Riem^{\psc}(X, \p X)
\end{gather*}
is called the space of \emph{well-adapted psc-metrics on $X$}.
\end{Definition}
 There is a further subspace of $\Riem_{\mathcal{S}}^{\psc}(X, \p X)$
 which we must introduce. Suppose $h\in \Riem^{\psc}(\beta X)$ is a
 psc-metric on the Bockstein manifold $\beta X$. Recall that
 $\Riem_{\mathcal{S}}^{\psc}(\p X)_{h}\subset
 \Riem_{\mathcal{S}}^{\psc}(\p X)$, denotes the pre-image,
 $\mathcal{B}^{-1}(h)$. We~now denote by
 $\Riem_{\mathcal{S}}^{\psc}(X, \p X)_h$, the space defined by
\begin{gather*}
\Riem_{\mathcal{S}}^{\psc}(X, \p X)_h:=\mathrm{res}^{-1}\big(\Riem_{\mathcal{S}}^{\psc}(\p X)_{h}\big)=\mathrm{res}^{-1}\big(\mathcal{B}^{-1}(h)\big).
\end{gather*}

All of this extends naturally to a definition of well-adapted metric
on the pseudo-manifold $X_{\Sigma}=X\cup_{\p X} -N(\beta X)$. A well
adapted psc-metric $g\in \Riem_{\mathcal{S}}^{\psc}(X, \p X)$, has
associated to it a~unique Riemannian submersion $p\colon \big(\p X, g^{\p}\big)\to
\big(\beta X, h^{\beta}\big)$ with a fibre $(L, \tau g_{L})$ for some constant
$\tau>0$. This determines
\begin{enumerate}\itemsep=0pt
\item[$(i)$] a cone metric, $g_{C(L)}$, as defined in~\eqref{conemetricdef}, although with $g_L$ now
replaced by $\tau g_{L}$ in that formula,
\item[$(ii)$] and an attaching metric $g_{\mathrm{att}(L)}$.
\end{enumerate}
In turn, we obtain a submersion metric on $N(\beta X)$ obtained by
equipping the fibres with $g_{\mathrm{att}(L)}\cup g_{C(L)}$. This
metric then attaches to $g$ in the obvious way to yield a metric on
$X_{\Sigma}$. This process can be thought of as a map
\begin{gather}\label{metricattach}
i_{\Sigma}\colon\ \Riem_{\mathcal{S}}^{\psc}(X, \p X)\longrightarrow \Riem(X_{\Sigma}).
\end{gather}
This map is easily seen to be a homeomorphism onto its image. This
leads to the following definition.
\begin{Definition}
The image of the map $i_{\Sigma}$, denoted $\Riem^{\psc}(X_{\Sigma})$,
is called the space of {\em well adapted psc-metrics on $X_{\Sigma}$}.
\end{Definition}
\noindent This space, $\Riem^{\psc}(X_{\Sigma})$, is the space we are
most interested in studying.
\begin{Remark}
It is worth emphasising that the space $\Riem_{\mathcal{S}}^{\psc}(\p
X)$ (as well the space $\Riem_{\mathcal{S}}^{\psc}(X, \p X)$) could be
empty even though the boundary $\p X$ has a psc-metric. A simple
example is the following. Let $\beta X= K3$. Then for any prime
generator $c\in H^{2}(K3;\Z)=\Z^{22}$ a corresponding map $c\colon K3\to
\CP^{\infty}$ gives a bundle $Z \to \beta X$, where $Z$ is a
simply-connected spin manifold which, of course, admits a
psc-metric. Moreover, $Z$ is a spin boundary: $Z=\p X$. However, no
psc-metric on $Z$ could be invariant under $S^1$-action, otherwise,
according to B{\'e}rard-Bergery~\cite[Theorem~C]{B-B}, it would imply
the existence of a psc-metric on $K3$.
\end{Remark}

As we mentioned above, there are two types of surgery that could be
performed on $X_{\Sigma}$: the first one on its resolution, the
interior of $X$, and the second one on the structure map $f\colon \beta X
\to BG$. We~consider the latter. Moreover, for now it is convenient to
cut the singularity out and to
work with a smooth manifold $X$ whose boundary is fibred
over $\beta X$ with fibre $L$. We~use the notation $(X,\beta X, f)$
for such manifold, where the boundary $\p X$ of $X$ is a total space
of the fibre bundle from the diagram:
\begin{gather*}
 \begin{diagram}
 \setlength{\dgARROWLENGTH}{1.6em}
\node{\p X}
 \arrow{e,t}{\hat f}
 \arrow{s,r}{p}
\node{E(L)}
 \arrow{s,r}{p_L}
 \\
\node{\beta X}
 \arrow{e,t}{f}
\node{BG.}
\end{diagram}
\end{gather*}
Here $p_L\colon E(L)\to BG$ is the universal fibre bundle with the fibre $L$
and the structure group $G$.

We recall the earlier bordism of smooth compact manifolds, $\bar B\colon B_0 \bord B_1 $. We~assume now that this is an elementary bordism
with $B_0 =\beta X$ and, as before, $\bar f\colon \bar B \to BG$ a map
such that $\bar f|_{B_0}=f_0=f$. We~will use the notation $\big(\bar
B,\bar f\big)\colon (B_0,f_0) \bord (B_1,f_1)$. Let $\bar p\colon \bar E \to \bar
B$ be a~corresponding fibre bundle with fibre $L$ and structure group~$G$. By construction, $\bar E|_{B_0}=E_0=\p X$, and $\bar
p|_{E_0}=p$. Then the manifold $\bar E$ gives a bordism $\bar E\colon E_0\bord E_1$, where $E_1=E|_{B_1}$. As~before, we assume that the
bordism $\bar B$ is equipped with collars $c_i\colon B_i\times [0,2)\to
 \bar{B}$, $i=0,1$, along the boundary~$\p \bar B$.

We now equip $B_0=\beta X$ with the metric $h^{\beta}_{0}=h^{\beta}
(=g_{\beta X})$. We~assume that there is a~psc-metric $\bar
h^{\beta}\in \Riem^{\psc}\big(\bar B\big)_{h^{\beta}_0,h^{\beta}_1}$, where
$h^{\beta}_1$ is a psc-metric on the manifold $B_1$. In particular,
this means that $c_i^{*}\bar h^{\beta}=h^{\beta}_i+{\rm d}t^{2}$, with
respect to the collars $c_i\colon B_i \times [0,2) \rh \bar{B}$, $i=0,1$
 near the boundary $\p \bar{B} = B_0\sqcup B_1$. Thus, the metric
 $\bar h^{\beta}$ provides a psc-bordism
\begin{gather*}
\big(\bar B,\bar f,\bar h^{\beta}\big)\colon\ \big(B_0,f_0,h^{\beta}_0\big) \bord
 \big(B_1,f_1,h^{\beta}_1\big).
\end{gather*}
The structure map $\bar f\colon \bar B \to BG$ determines a
bordism $\bar E\colon E_0\bord E_1$, where $\bar E$ is a pull-back of the
universal $(L,G)$-fibration:
\begin{gather*}
 \begin{diagram}
 \setlength{\dgARROWLENGTH}{1.95em}
\node{\bar E}
 \arrow{e,t}{\hat{\bar f}}
 \arrow{s,l}{\bar{p}}
\node{E(L)}
 \arrow{s}
 \\
\node{\bar B}
 \arrow{e,t}{\bar f}
\node{BG}
\end{diagram}
 \end{gather*}
 with $\bar E|_{B_0}=E_0$ and $\bar E|_{B_1}=E_1$. Now we glue the
 manifolds $X$ and $\bar E$ (again, making use of~col\-lars near their
 boundaries) to obtain the manifold $X_1= X\cup_{\p X} \bar E$ with
 boundary $\p X_1 = E_1$, which is the total space of the
 $(L,G)$-fibration $p_1\colon \p X_1\to B_1$. We~denote $\beta X_1 = B_1$. We~denote by $\mathcal{S}_1$, the analogue of the map $\mathcal{S}$
 in~\eqref{Slift}, for the bundle, $p_1\colon \p X_1\to B_1$ and by $
 \Riem_{\mathcal{S}_1}^{\psc}(X_1, \p X_1)_{h^{\beta}_1}$, the
 analogous space of psc-metrics on $X_1$ with boundary metric a
 psc-submersion which restricts on the base as the psc-metric
 $h^{\beta}_1$. We~will shortly employ the psc-submersion
 construction from Lemma~\ref{pscliftlemma} to obtain a
 map
 \begin{gather*}
 \Riem_{\mathcal{S}}^{\psc}(X, \p
 X)_{h^{\beta}_0}\longrightarrow \Riem_{\mathcal{S}_1}^{\psc}(X_1, \p
 X_1)_{h^{\beta}_1},
 \end{gather*}
 an analogue of the map in~\eqref{gluemap-1}. In preparation for this, there are some objects
 we must re-introduce.

 In the preamble to Lemma~\ref{pscliftlemma}, we considered spaces of
 horizontal distributions, $\mathcal{HD}(p_0)$ and
 $\mathcal{HD}_{\mathrm{adm}}(\bar{p})$. In the latter case, we
 recall that distributions satisfy a product structure near the
 boundary components $E_{i}$. It will be useful for us to consider
 the space of maps from $\mathcal{HD}(p_0)$ to
 $\mathcal{HD}_{\mathrm{adm}}(\bar{p})$, which we denote
 $\mathcal{F}(\mathcal{HD}(p_0),
 \mathcal{HD}_{\mathrm{adm}}(\bar{p}))$, under the usual compact-open
 topology. In particular, we will be interested in maps
 $\xi\colon\mathcal{HD}(p_0)\to\mathcal{HD}_{\mathrm{adm}}(\bar{p})$ which
 satisfy the condition that $\xi(\mathcal{H})|_{E_0}=\mathcal{H}$.
 Such maps are easy to construct. For~example, fix a disribution
 $\bar{\mathcal{H}}$ on~$\bar{E}|_{\bar{B}\setminus c_{0}(B\times
 [0,1])}$ which takes a product structure,
 $\mathcal{H}_{1}\times\mathbb{R}$, on $\bar{E}|_{c_{0}(B\times
 [1,2))}$. Then define $\xi(\mathcal{H})$ as
 $\mathcal{H}\times\mathbb{R}$ on $\bar{E}|_{c_{0}(B\times [0,1]}$,
 $\bar{\mathcal{H}}$ on $\bar{E}|_{\bar{B}\setminus c_{0}(B\times
 [0,2))}$ while continuously transitioning bet\-ween $\mathcal{H}$
 and $\mathcal{H}_1$ along $E|_{{c_{0}(B\times [1,2))}}$. We~denote by $\mathcal{F}_0(\mathcal{HD}(p_0),
 \mathcal{HD}_{\mathrm{adm}}(\bar{p}))$, the subspace of~$\mathcal{F}(\mathcal{HD}(p_0),
 \mathcal{HD}_{\mathrm{adm}}(\bar{p}))$, defined by
 \begin{gather*}
 \mathcal{F}_0\big(\mathcal{HD}(p_0), \mathcal{HD}_{\mathrm{adm}}(\bar{p})\big):=\big\{\xi\in
 \mathcal{F}(\mathcal{HD}(p_0),
 \mathcal{HD}_{\mathrm{adm}}(\bar{p}))\colon \xi(\mathcal{H})|_{E_0}=\mathcal{H}\big\}.
 \end{gather*}
 The proof of the
 following proposition is an elementary exercise.
\begin{Proposition}
The spaces $\mathcal{F}(\mathcal{HD}(p_0),
\mathcal{HD}_{\mathrm{adm}}(\bar{p}))$ and
$\mathcal{F}_0(\mathcal{HD}(p_0),
\mathcal{HD}_{\mathrm{adm}}(\bar{p}))$ are convex.
\end{Proposition}
We next recall that a metric $g\in\Riem_{\mathcal{S}}^{\psc}(X, \p
X)_{h^{\beta}_0}$ restricts on the boundary $\p X$ as a submersion
metric, $g^{\p}=g|_{\p X}$ over the base metric $ h^{\beta}_0$, with
fibre metric $\tau_{0}(g) g_{L}$ (for some fibre metric scaling
constant $\tau_{0}(g)>0$) and horizontal distribution
$\mathcal{H}_{0}(g)$. Both the fibre metric scaling constant and the
horizontal distribution vary continuously with respect to $g$. We~now
choose a map $\xi\colon\mathcal{HD}(p)\longrightarrow
\mathcal{HD}_{\mathrm{adm}}(\bar{p})$ from the space
$\mathcal{F}_0(\mathcal{HD}(p_0),
\mathcal{HD}_{\mathrm{adm}}(\bar{p}))$. Thus, for any
$g\in\Riem_{\mathcal{S}}^{\psc}(X, \p X)_{h^{\beta}_0}$, we obtain a
distribution $\xi(\mathcal{H}_{0}(g))$ on the total space $\bar{E}$,
extending $\mathcal{H}_{0}(g)$ from~$E_0$. This allows us to specify
a map
\begin{align}\label{gluemap-1a}
 \mu_{(\bar B,\bar f, \bar h^{\beta})}\colon\ \Riem_{\mathcal{S}}^{\psc}(X, \p X)_{h^{\beta}_0}
 &\longrightarrow
 \Riem_{\mathcal{S}_1}^{\psc}(X_1, \p X_1)_{h^{\beta}_1},\nonumber
 \\
 g &\longmapsto g\cup \bar g^{\p},
\end{align}
where $\bar g^{\p}$ is the metric obtained in Lemma~\ref{pscliftlemma}
above with respect to the triple of base metric $\bar h^{\beta}\in
\Riem^{\psc}\big(\bar B\big)_{h^{\beta}_0,h^{\beta}_1}$, admissible horizontal
distribution $\xi(\mathcal{H}_{0}(g))$ and fibre scaling constant
$\bar{\tau}(\bar h^{\beta}, \xi(\mathcal{H}_{0}(g)))$ (as defined in~\eqref{taumin}). The following is an easy consequence of Lemma
\ref{pscliftlemma}.
\begin{Corollary}
The homotopy type of the map $\mu_{(\bar B,\bar f, \bar h^{\beta})}$
is invariant of the psc-isotopy type of~the base metric $\bar
h^{\beta}$, as well as the choice of distribution map $\xi\in
\mathcal{F}_0(\mathcal{HD}(p_0),
\mathcal{HD}_{\mathrm{adm}}(\bar{p}))$.
\end{Corollary}
As our work concerns only homotopy type, we feel justified
in suppressing the role of $\xi$, in~the notation of the map~\eqref{gluemap-1a} above.

Now we are ready to state our main technical result which is similar
to Theorem~\ref{thm-boundary}.
\begin{Theorem}\label{thm-fb-sing}
 Let $(X,\beta X, f)$ be a manifold with fibred singularities, i.e., the boundary $\p X$ is a total space of an $(L,G)$-fibration $\p
 X\to \beta X$ given by the structure map $f\colon \beta X \to BG$, where
 $\dim X = n$, $\dim L= \ell$. Furthermore, we assume $\big(\bar B,\bar
 f\big)\colon (B_0,f_0)\bord (B_1,f_1)$ is an elementary bordism with $p,q\geq
 2$, where $B_0=\beta X$. Then for any psc-metric $h^{\beta}_0$ on
 $B_0$ there exist psc-metrics~$h^{\beta}_1$ on~$B_1$ and $\bar
 h^{\beta}\in \Riem^{\psc}\big(\bar B\big)_{h^{\beta}_0,h^{\beta}_1}$ such
 that the map
\begin{align*}
 \mu_{(\bar B,\bar f, \bar h^{\beta})}\colon\ & \Riem_{\mathcal{S}}^{\psc}(X,\p X)_{h^{\beta}_0} \longrightarrow \Riem_{\mathcal{S}_1}^{\psc}(X_1,\p X_1)_{h^{\beta}_1}, \qquad
 g \longmapsto g\cup \bar g^{\p},\\
 & X_1=X\cup_{\p X=E_0} \bar E, \qquad \beta X_1=
 E_1,
\end{align*}
defined as in \eqref{gluemap-1a} above, is a weak homotopy equivalence.
\end{Theorem}
\begin{Remark}
The construction of the metrics $h^{\beta}_1$ on $B_1$ and $\bar
h^{\beta}$ on $\bar{B}$ used in this theorem, follows from surgery
techniques pioneered by Gromov and Lawson~\cite{GL}, and later
generalised by others. It should be pointed out however that, given
such a metric, $h^{\beta}_1$, the theorem works equally well if
$h^{\beta}_1$ is replaced by any psc-metric which is psc-isotopic to
$h^{\beta}_1$, i.e., lies in the same path component of the space
$\Riem^{\psc}(B_1)$.
 \end{Remark}
\section{Proof of Theorem~\ref{thm-fb-sing}}
\subsection{Some standard metric constructions}
Here we briefly a recall a couple of standard metric
constructions. These constructions are discussed in detail in
\cite[Section 5]{W3}.

We fix some constants $\delta>0$ and $\lambda\geq 0$. Then a
\emph{$(\delta$-$\lambda)$-torpedo metric} on the disk $D^{n}$, deno\-ted~$g_{\tor}^{n}(\delta)_{\lambda}$, is a psc-metric which roughly takes
the form a round hemisphere of radius $\delta>0$ near the centre
before transitioning into a cylinder of radius $\delta$ and length
$\lambda\geq 0$ near the boundary; see first image in
Figure~\ref{Torpcorners-a} below.

Letting $D_{+}^{n}$ denote the upper hemi-disk, we obtain the metric
$g_{\tor+}^{n}(\delta)_{\lambda}:=g_{\tor}^{n}(\delta)_{\lambda}|_{D_{+}^{n}}$;
see second image in Figure~\ref{Torpcorners-a}. We~call
$g_{\tor+}^{n}(\delta)_{\lambda}$ a \emph{half-torpedo metric}. Let
$\lambda_2>0$ be some constant. Next, we consider the cylinder
$D^{n-1}\times [0, \lambda_2]$ equipped with the metric
$g_{\tor}^{n-1}(\delta)_{\lambda_1}+{\rm d}t^{2}$ and attach a half-disk
$D_{+}^{n}$ with half-torpedo metric
$g_{\tor+}^{n}(\delta)_{\lambda_1}$ along $D^{n-1}\times \{0\}$. We~denote the resulting Riemannian manifold by
$\big(D_{\mathrm{stretch}}^{n},\hat{g}_{\tor}^{n}(\delta)_{\lambda_1,
 \lambda_2}\big)$. This is depicted in the third image in~Figure~\ref{Torpcorners-a}.

Typically, we will not care so much about the $\lambda_2$-parameter
but only $\lambda_1$ which we regard as the vertical height of this
metric. Moreover, we will usually be interested in the case when
$\lambda_1=1$ and when $\delta=1$. With this in mind we make use of
the following notational simplifications.
\begin{gather*}
g_{\tor}^{n}:=g_{\tor}^{n}(1)_{1}.
\\
g_{\tor+}^{n}:=g_{\tor+}^{n}(1)_{1}.
\\
\hat{g}_{\tor}^{n}(\delta)_{\lambda_1}:=
 \hat{g}_{\tor}^{n}(\delta)_{\lambda_1,
 \lambda_2},\qquad \text{where}\quad \lambda_2\quad \text{is arbitrary}.
\\
 \hat{g}_{\tor}^{n}:=\hat{g}_{\tor}^{n}(1)_{1}=\hat{g}_{\tor}^{n}(1)_{1,1}.
\end{gather*}
The following proposition follows immediately from~\cite[Proposition~3.1.6]{W3}.

\begin{Proposition}
Let $n\geq3$, $\delta>0$ and $\lambda, \lambda_1, \lambda_2\geq 0$.
\begin{enumerate}\itemsep=0pt
\item[$(i)$] The metrics $g_{\tor}^{n}(\delta)_{\lambda}$,
 $g_{\tor+}^{n}(\delta)_{\lambda}$ and
 $\hat{g}_{\tor}^{n}(\delta)_{\lambda_1, \lambda_2}$ have positive
 scalar curvature.
\item[$(ii)$] For any constant $b\geq 0$ and any $\lambda,
 \lambda_1, \lambda_2\geq 0$, there exists $\delta>0$ so that the
 scalar curvature of the metrics $g_{\tor}^{n}(\delta)_{\lambda}$,
 $g_{\tor+}^{n}(\delta)_{\lambda}$ and
 $\hat{g}_{\tor}^{{n}}(\delta)_{\lambda_1, \lambda_2}$ is bounded
 below by $b$.
\end{enumerate}
\end{Proposition}
\begin{figure}[!htbp]
\vspace{3cm}
\hspace{10mm}
\begin{picture}(0,0)
\includegraphics{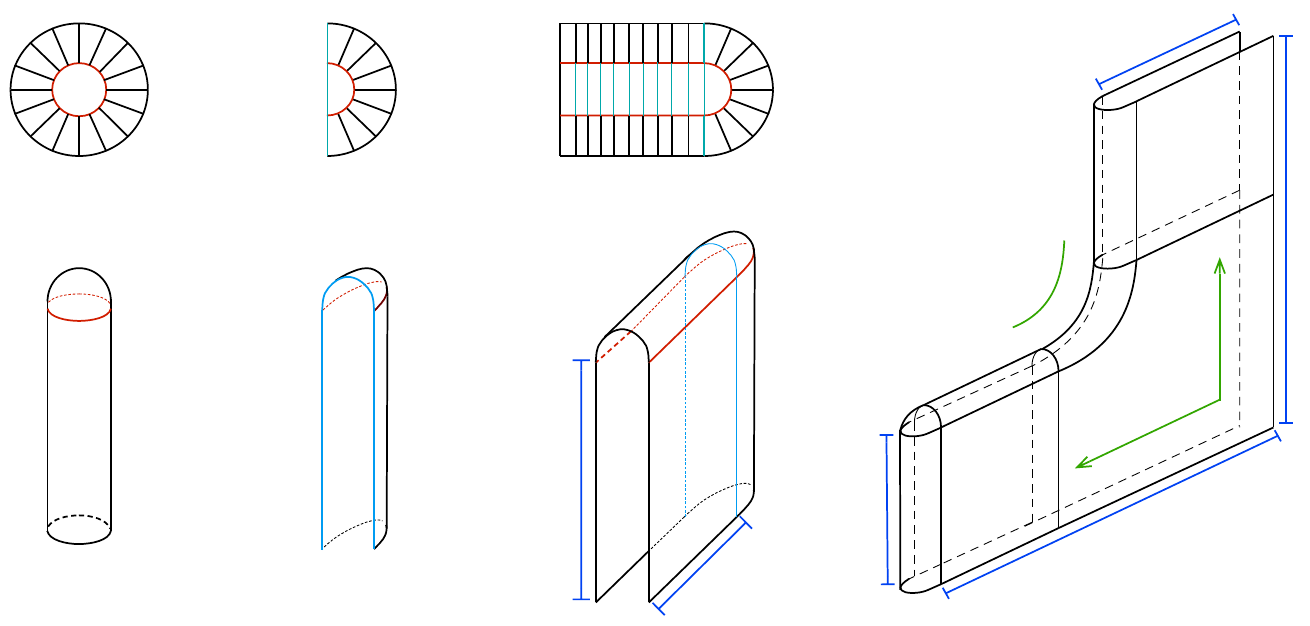}%
\end{picture}
\setlength{\unitlength}{3947sp}
\begingroup\makeatletter\ifx\SetFigFont\undefined%
\gdef\SetFigFont#1#2#3#4#5{%
 \reset@font\fontsize{#1}{#2pt}%
 \fontfamily{#3}\fontseries{#4}\fontshape{#5}%
 \selectfont}%
\fi\endgroup%
\begin{picture}(5079,1559)(1902,-7227)
\put(4400,-6500){\makebox(0,0)[lb]{\smash{{\SetFigFont{10}{8}{\rmdefault}{\mddefault}{\updefault}{\color[rgb]{0,0,1}$\lambda_1$}%
}}}}
\put(5220,-7000){\makebox(0,0)[lb]{\smash{{\SetFigFont{10}{8}{\rmdefault}{\mddefault}{\updefault}{\color[rgb]{0,0,1}$\lambda_2$}%
}}}}
\put(7500,-5700){\makebox(0,0)[lb]{\smash{{\SetFigFont{10}{8}{\rmdefault}{\mddefault}{\updefault}{\color[rgb]{0,0.6,0}$t_1$}%
}}}}
\put(7100,-6300){\makebox(0,0)[lb]{\smash{{\SetFigFont{10}{8}{\rmdefault}{\mddefault}{\updefault}{\color[rgb]{0,0.6,0}$t_2$}%
}}}}
\put(6600,-5600){\makebox(0,0)[lb]{\smash{{\SetFigFont{10}{8}{\rmdefault}{\mddefault}{\updefault}{\color[rgb]{0,0.6,0}$\frac{\pi}{2}\Lambda$}
}}}}
\put(5950,-6660){\makebox(0,0)[lb]{\smash{{\SetFigFont{10}{8}{\rmdefault}{\mddefault}{\updefault}{\color[rgb]{0,0,1}$l_1$}%
}}}}
\put(7300,-6750){\makebox(0,0)[lb]{\smash{{\SetFigFont{10}{8}{\rmdefault}{\mddefault}{\updefault}{\color[rgb]{0,0,1}$l_2$}%
}}}}
\put(8100,-5350){\makebox(0,0)[lb]{\smash{{\SetFigFont{10}{8}{\rmdefault}{\mddefault}{\updefault}{\color[rgb]{0,0,1}$l_3$}%
}}}}
\put(7370,-4350){\makebox(0,0)[lb]{\smash{{\SetFigFont{10}{8}{\rmdefault}{\mddefault}{\updefault}{\color[rgb]{0,0,1}$l_4$}%
}}}}
\end{picture}%
\caption{The metrics $g_{\tor}^{n}(\delta)_{\lambda}$,
 $g_{\tor+}^{n}(\delta)_{\lambda}$ and
 $\hat{g}_{\tor}^{n}(\delta)_{\lambda_1, \lambda_2}$ (bottom) on the
 manifolds $D^{n}$, $D_{+}^{n}$ and~$D_{\mathrm{stretch}}^{n}$ (top)
 followed by the boot metric $g_{\boot}^{n}(\delta)_{\Lambda,
 \bar{l}}$.}
\label{Torpcorners-a}
\end{figure}
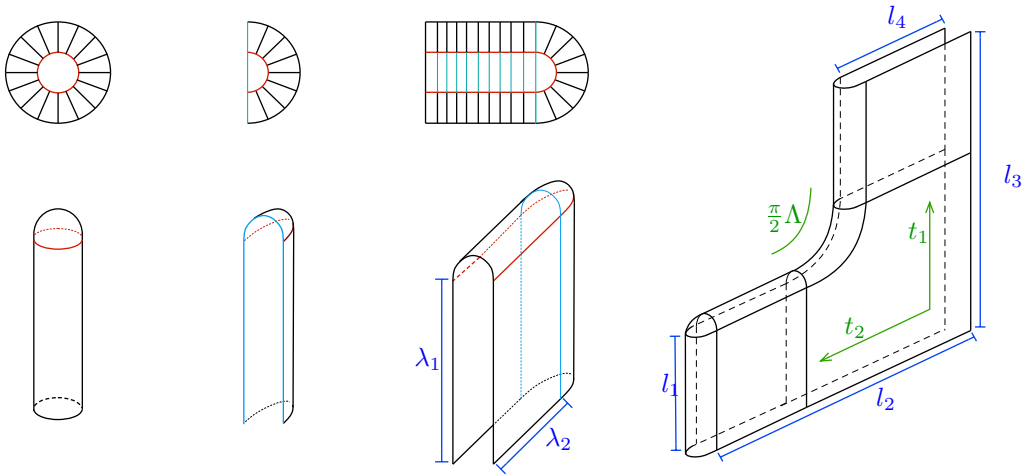
We now consider product metrics
$g_{\tor}^{n-1}(\delta)_{\lambda}+{\rm d}t^{2}$ on the cylinder
$D^{n-1}\times I$. It is shown in~\cite[Section 5]{W3}, provided
$n\geq 4$, that any such product metric
$g_{\tor}^{n-1}(\delta)_{\lambda}+{\rm d}t^{2}$ can be moved by isotopy
through psc-metrics to a particular psc-metric called a {\em
 $\delta$-boot metric}. A detailed account of how such metrics are
constructed can be found in~\cite{W3}. Here we will provide brief
description.
\begin{enumerate}\itemsep=0pt
\item[$(i)$] Beginning with some torpedo metric,
 $g_{\tor}^{n-1}(\delta)_{\lambda}$, trace out a cylinder of torpedo
 metrics before bending the cylinder around an angle of
 $\frac{\pi}{2}$ to finish as a Riemannian cylinder (perpendicular to
 the first one) in the direction suggested by the rightmost image of
 Figure~\ref{Torpcorners-a}. The resulting object has two cylindrical
 ends, one of the form ${\rm d}t_{1}^{2}+g_{\tor}^{n-1}(\delta)_{\lambda}$
 and the other ${\rm d}t_{2}^{2}+g_{\tor}^{n-1}(\delta)_{\lambda}$, where
 $t_1$ and $t_2$ are orthogonal coordinates depicted in
 Figure~\ref{Torpcorners-a}.
\item[$(ii)$] In order to ensure the resulting metric has positive
 scalar curvature, the bending is controlled by a paramater
 $\Lambda>0$. Essentially, the bending takes place along a
 quarter-circle of~radius $\Lambda>0$. A sufficiently large choice of
 $\Lambda$ ensures that curvature arising from the bend is small and
 positivity of the scalar curvature arising from the torpedo factor
 dominates.
\item[$(iii)$] Away from the ``caps" of the torpedos, this metric takes
 the form ${\rm d}t_{1}^{2}+{\rm d}t_{2}^{2}+\delta^{2}{\rm d}s_{n-1}^{2}$. This part
 can easily be extended to incorporate the corner depicted in the
 rightmost image of Figure~\ref{Torpcorners-a} and so that the necks of
 the torpedo ``ends" have any desired lengths, $l_1$ and $l_4$
 (determining along with $\Lambda$ the distances $l_2$ and $l_3$).
\item[$(iv)$] Finally, we smoothly ``cap-off" the cylindrical end which
 takes the form ${\rm d}t_{1}^{2}+g_{\tor}^{n-1}(\delta)_{l_1}$, by
 attaching a half-torpedo metric, $g_{\tor+}^{n}(\delta)_{l_1}$. This
 is the so-called ``toe" of the boot metric.
\end{enumerate}
\noindent The resulting metric is denoted
$g_{\boot}^{n}(\delta)_{\Lambda, \bar{l}}$, where $\Lambda>0$ is the
bending constant discussed above and $\bar{l}=(l_1, l_2, l_3,
l_4)\in(0,\infty)^{4}$ determines the various neck-lengths. As~we
mentioned above, $\Lambda$~may need to be chosen to be large (although
can always be found) and will depend on $\delta$. While the choices of
$l_1$ and $l_4$ are arbitrary, the constants $l_2$ and $l_3$ are
determined by $\Lambda$, $l_1$ and~$l_4$.

\subsection{Back to the proof of Theorem~\ref{thm-fb-sing}}
The proof follows from that of~\cite[Theorem~A]{W2}. We~will provide
a brief review of the main steps of that proof and show that it goes
through perfectly well in our case. The strategy is to decompose the
map
\begin{gather*}
 \mu_{(\bar B,\bar f, \bar
 h^{\beta})}\colon\ \Riem_{\mathcal{S}}^{\psc}(X, \p X)_{h^{\beta}_0} \longrightarrow
 \Riem_{\mathcal{S}_1}^{\psc}(X_1, \p X_1)_{h^{\beta}_1},
\end{gather*}
into a composition of three maps as shown in the commutative diagram
below:
\begin{gather}\label{bdydiag-1}
\xymatrix{
 & \Riem_{\mathcal{S}}^{\psc}(X, \p X)_{h^{\beta}_0} \ar@{->}[d]^{\mu_{\boot}}
 \ar@{->}[r]^{\mu_{(\bar B,\bar f, \bar
 h^{\beta})}}&
 \Riem_{\mathcal{S}_1}^{\psc}(X_1, \p X_1)_{h^{\beta}_1}\\
 & \Riem_{\mathcal{S}-\boot}^{\psc}(X,\p X)_{h^{\beta}_{\std}} \ar@{->}[r]^{\mu_{\Estd}} &
 \Riem_{\mathcal{S}_1-\Estd}^{\psc}(X_1,\p X_1)_{h^{\beta}_1}. \ar@{^{(}->}[u]^{}
}
\end{gather}
Here the right vertical map denotes inclusion. We~will define the
spaces $\Riem_{\mathcal{S}-\boot}^{\psc}(X,\p X)_{h^{\beta}_{\std}}$
and $\Riem_{\mathcal{S}_1-\Estd}^{\psc}(X_1, \p X_1)_{h^{\beta}_1}$
and the maps $\mu_{\boot}$ and $\mu_{\Estd}$ in due course. The point
is to show that each of these maps is a weak homotopy equivalence.

We denote by $k=n-\ell-1=\dim \beta X= B_0$. We~consider carefully the
elementary bordism $\big(\bar B,\bar f\big)\colon (B_0,f_0)\bord (B_1,f_1)$. The
manifold $\bar B$ is given by attaching a handle $D^{p+1}\times
D^{q+1}$ to $B_0$ along the embeddings $\phi\colon S^{p}\times D^{q+1}\rh
B_0$, where $p+q+1=k$, $q\geq 2$. We~would like to have some
flexibility for the embedding $\phi$. We~introduce the following
family of rescaling maps:
\begin{gather*}
\begin{split}
\sigma_\rho\colon\ S^{p}\times D^{q+1}&\longrightarrow S^{p}\times D^{q+1},\\
(x,y)&\longmapsto (x, \rho y),
\end{split}
\end{gather*}
where $\rho\in(0,1]$. We~set
\begin{gather*}
\phi_{\rho}:=\phi\circ\sigma_{\rho}\colon\ S^{p}\times D^{q+1}\rh B_0
\end{gather*}
and $N_{\rho}:=\phi_{\rho}(S^{p}\times D^{q+1})$, abbreviating
$N:=N_{1}$ and $\phi:=\phi_1$. Let $T_{\phi}$ be the trace of the
surgery on $B_0$ with respect to $\phi$. We~denote by
$\Riem_{\std}^{\psc}(B_0)$, the space defined as follows:
\begin{gather*}
\Riem_{\std}^{\psc}(B_0):=\big\{g\in \Riem^{\psc}(B_0)\colon
\phi_{\frac{1}{2}}^{*}g={\rm d}s_{p}^{2}+g_{\tor}^{q+1} \text{ on }
S^{p}\times D^{q+1}\big\}.
\end{gather*}
According to Chernysh's theorem~\cite{Ch,EF}, the inclusion
\begin{gather*}
 \Riem_{\std}^{\psc}(B_0)\subset \Riem^{\psc}(B_0),
\end{gather*}
is a weak homotopy equivalence. A major step in the proof of this
theorem is the fact (which follows easily enough from the original
Gromov--Lawson construction in~\cite{GL}) that for any psc-metric
$h^{\beta}\in \Riem^{\psc}(B_0)$, there is an isotopy $h_t^{\beta}$,
$t\in I$ of metrics in $\Riem^{\psc}(B_0)$ connecting
$h_0^{\beta}=h^{\beta}$ to a psc-metric $h_{\std}^{\beta}\in
\Riem_{\std}^{\psc}({B_0})$. By a well known argument (see~\cite[Lemma
 2.3.2]{W2}), this isotopy gives rise to a concordance:
$\bar{h}_{\con}^{\beta}$ on $B_0\times [0,\lambda+2]$ for some
$\lambda>0$ which takes the form of~product metrics:
\begin{gather*}
 h^{\beta} +{\rm d}t^{2}\quad\text{on}\quad B_0 \times [\lambda+1, \lambda+2]\qquad
 \text{and}\qquad g_{\std}+{\rm d}t^{2}\quad \text{on}\quad B_0\times[0,1].
\end{gather*}
Note that on the slice $N_{\frac{1}{2}}\times [0,1]$, the metric
$\bar{h}_{\con}^{\beta}$ pulls back to a metric of the form
\begin{gather}\label{new03}
 {\rm d}s_{p}^{2}+g_{\tor}^{q+1}+{\rm d}t^{2}.
\end{gather}
 Making use of~\cite[Lemma 5.2.5]{W2} we can perform an isotopy of
 the metric $\bar{h}_{\con}^{\beta}$, adjusting only on~$N_{\frac{1}{2}}\times [0,1]$, to replace the
 $g_{\tor}^{q+1}+{\rm d}t^{2}$ factor in~\eqref{new03} with
 $g_{\boot}^{q+2}(1)_{\Lambda, \bar{l}}$ for some appropriately large
 $\Lambda>0$ and with $\bar{l}$ satisfying $l_1=l_4=1$. We~denote the
 resulting psc-metric $\bar{h}_{\pre}^{\beta}$ on~$B_0 \times [0,
 \lambda+2]$. We~consider $B_0 \times [0, \lambda+2]$ as a long
 collar of $\bar B$ and assume that the map $\bar f$ restricted to~$B_0 \times [0, \lambda+2]$ is given by $\bar f(x,t) = f_0(x)$. Let
 $\bar E_0$ be a manifold given by pulling back the fibre bundle
\begin{gather*}
 \begin{diagram}
 \setlength{\dgARROWLENGTH}{1.95em}
\node{\bar E_0}
 \arrow{e,t}{\hat{\bar f}_0}
 \arrow{s,l}{p}
\node{E(L)}
 \arrow{s}
 \\
\node{B\times [0, \lambda+2]}
 \arrow{e,t}{\bar f_0}
\node{BG,}
\end{diagram}
\end{gather*}
where $\bar f_0$ is a restriction of $\bar f$.

We now use the metric $\bar{h}_{\pre}^{\beta}$ on $B_0 \times [0,
 \lambda+2]$ to extend the metric $g^{\p}$ from the boundary \mbox{$E_0=\p
X$} to a positive scalar curvature submersion metric $\bar
g^{\p}_{\pre}$ on the total space of this bundle,~$\bar E_{0}$. To~do
this, we follow the method described in Lemma~\ref{pscliftlemma} with
respect to the product distribution $\mathcal{H}_{0}\times \mathbb{R}$
to extend $g^{\p}$ to a psc submersion metric, on $\bar{E}_0\cong
E_0\times [0, \lambda+2]$. We~denote the resulting metric, $\bar
g^{\p}_{\pre}$. Note, near $E_0\times \{\lambda+2\}$, this is a
Riemannian cylinder of a~psc-submersion metric over $\big(B_{0},
h^{\beta}_{\std}\big)$ with the fibre $(L, \bar{\tau} g_L)$ for
$\bar{\tau}=\bar{\tau}\big(\bar{h}_{\pre}^{\beta}, \mathcal{H}_{0}\times
\mathbb{R}\big)$, as defined in~\eqref{taumin}.

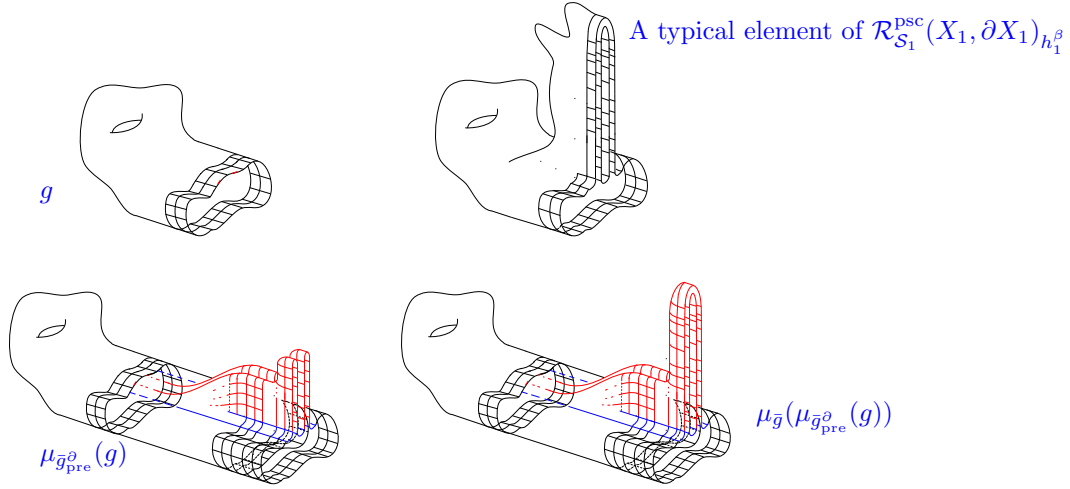
\begin{figure}[!htbp]
\vspace{32mm}
\hspace{10mm}
\begin{picture}(0,0)
\includegraphics{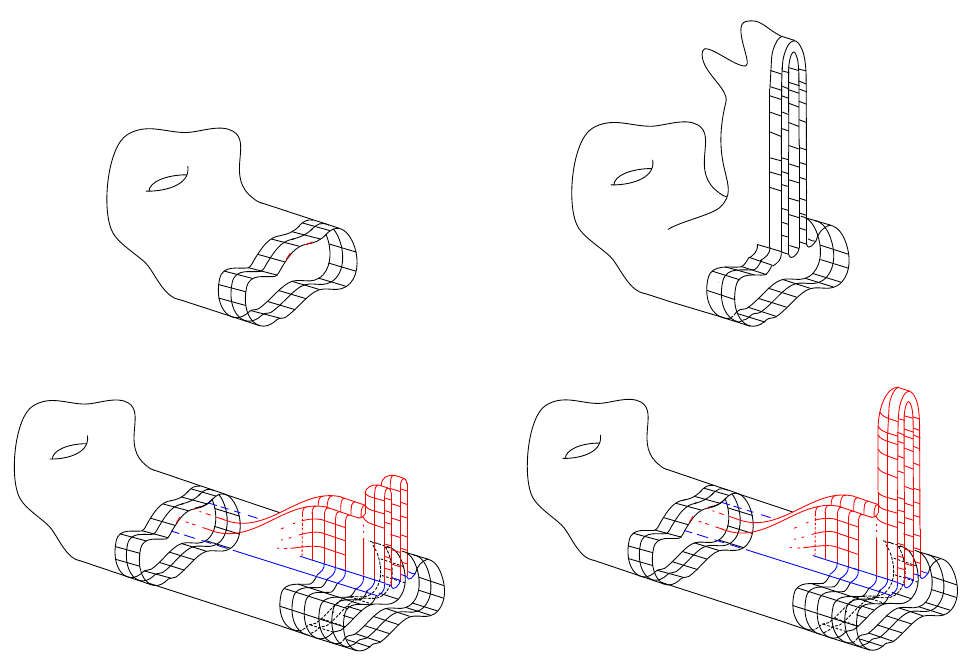}%
\end{picture}
\setlength{\unitlength}{3947sp}
\begingroup\makeatletter\ifx\SetFigFont\undefined%
\gdef\SetFigFont#1#2#3#4#5{%
 \reset@font\fontsize{#1}{#2pt}%
 \fontfamily{#3}\fontseries{#4}\fontshape{#5}%
 \selectfont}%
\fi\endgroup%
\begin{picture}(5079,1559)(1902,-7227)
\put(2100,-7000){\makebox(0,0)[lb]{\smash{{\SetFigFont{10}{8}{\rmdefault}{\mddefault}{\updefault}{\color[rgb]{0,0,1}$\mu_{\bar{g}_{\pre}^{\p}}(g)$}%
}}}}
\put(6600,-6750){\makebox(0,0)[lb]{\smash{{\SetFigFont{10}{8}{\rmdefault}{\mddefault}{\updefault}{\color[rgb]{0,0,1}$\mu_{\bar{g}}(\mu_{\bar{g}_{\pre}^{\p}}(g))$}%
}}}}
\put(2100,-5350){\makebox(0,0)[lb]{\smash{{\SetFigFont{10}{8}{\rmdefault}{\mddefault}{\updefault}{\color[rgb]{0,0,1}$g$}%
}}}}
\put(5800,-4350){\makebox(0,0)[lb]{\smash{{\SetFigFont{10}{8}{\rmdefault}{\mddefault}{\updefault}{\color[rgb]{0,0,1}A typical element of $\Riem_{\mathcal{S}_1}^{\psc}(X_1, \p X_1)_{h^{\beta}_1}$}%
}}}}
\end{picture}%
\caption{Representative elements of the spaces from the commutative
 diagram (\ref{bdydiag-1}) above in the case when $L$ is a point.}
\label{thmAdiagram-1}
\end{figure}
For any element $g\in \Riem_{\mathcal{S}}^{\psc}(X, \p
X)_{h^{\beta}_0}$, the metric $g\cup_{\p X} \bar g^{\p}_{\pre}$ on
$X\cup \bar E_0$ (obtained by the obvious gluing) is denoted by
$g_{\std}$ and is an element of the space
$\Riem_{\mathcal{S}}^{\psc}(X,\p X)_{h^{\beta}_{\std}}$. This gives a
map
\begin{gather*}
\mu_{\bar g^{\p}_{\pre}}\colon\ \Riem_{\mathcal{S}}^{\psc}(X, \p X)_{h^{\beta}_0}
\to \Riem_{\mathcal{S}}^{\psc}(X,\p X)_{h^{\beta}_{\std}}.
\end{gather*}
We denote $\Riem_{\mathcal{S}-\boot}^{\psc}(X,\p
X)_{h^{\beta}_{\std}}:=\mathrm{Im}(\mu_{\bar g^{\p}_{\pre}} )$. This
new metric is depicted in the bottom left of~Figure~\ref{thmAdiagram-1},
with the original metric $g$ depicted in the top left. For~clarity,
this figure depicts only the case when $L$ is a point. Lemma 6.5.5 of
\cite{W2}, consolidating work from previous sections, shows that in
the case when $L$ is a point (and so $\p X=\beta X$), the map
$\mu_{\bar{g}_{\pre}}$ is a weak homotopy equivalence. This is
demonstrated by constructing certain isotopies of psc-metrics on the
cylinder $B_{0}\times [0,\lambda+2]$. In our case, where $L$ is not
simply a point, we must lift such a compact family of psc-metrics on
$B_{0}\times [0,\lambda+2]$ to a corresponding compact family of
psc-metrics on $\bar{E}_0$. This is done using Lemma
\ref{pscliftlemma}, with respect to the horizontal distribution
$\mathcal{H}_{0}\times \mathbb{R}$ and a sufficiently small scaling
constant $\bar{\tau}_{\min}>0$.

Turning our attention momentarily to the space $X_1=X\cup\bar{E}$,
recall that, though suppressed in the notation, the map $\mu_{(\bar
 B,\bar f, \bar h^{\beta})}$ implicitly associates to the bundle,
$\bar{p}\colon\bar{E}\to\bar{B}$, an admissible horizontal distribution,
$\bar{\mathcal{H}}$ extending $\mathcal{H}_{0}$ (earlier we denoted
this $\xi(\mathcal{H}_{0})$).

Consider now an element of the space
$\Riem_{\mathcal{S}-\boot}^{\psc}(X,\p
X)_{h^{\beta}_{\std}}:=\mathrm{Im}(\mu_{\bar g^{\p}_{\pre}} )$. Such
an element has, near the boundary, a ``standard piece" where the
submersion metric $\bar g_{\pre}^{\p}$ restricts over a~region where
the base metric takes the form
${\rm d}s_{p}^{2}+g_{\boot}^{q+2}(1)_{\Lambda, \bar{l}}$. Replacing this
piece, on the base, with $g_{\tor}^{p+1}+g_{\tor}^{q+1}$ near the
boundary, determines a base metric on $\bar{B}$. This is exactly what
is done in~\cite[Theorem~A]{W2}, determining a homeomorphism onto its
image in the relevant spaces of~psc-metrics on the bases. Using Lemma
\ref{pscliftlemma}, with respect to the distribution
$\bar{\mathcal{H}}$ allows us to lift the resulting base metric on
$\bar{B}$ to a psc-submersion metric on $\bar{E}$. This determines a
map
\begin{gather*}
 \mu_{\Estd}\colon\ \Riem_{\mathcal{S}-\boot}^{\psc}(X,\p X)_{h^{\beta}_{\std}} \to
 \Riem_{\mathcal{S}_1}^{\psc}(X_1,\p X_1)_{h^{\beta}_1}.
\end{gather*}
Replacing the codomain of this map by its image, denoted
\begin{gather*}
\Riem_{\mathcal{S}_1-\Estd}^{\psc}(X_1, \p X_1)_{h^{\beta}_1}\subset
\Riem_{\mathcal{S}_1}^{\psc}(X_1,\p X_1)_{h^{\beta}_1},
\end{gather*}
we obtain the lower horizontal map in diagram (\ref{bdydiag-1}). A
typical element in the image of this map is depicted in the lower
right of Figure~\ref{thmAdiagram-1}. As~with the base case in
\cite[Theorem~A]{W2}, this lower horizontal map is demonstrably a
homeomorphism.

It remains to show that the inclusion
$\Riem_{\mathcal{S}_1-\Estd}^{\psc}(X_1, \p X_1)_{h^{\beta}_1} \subset
\Riem_{\mathcal{S}_1}^{\psc}(X_1,\p X_1)_{h^{\beta}_1}$ is a weak
homotopy equivalence. Note that the notaiton ``$\Estd$" used in
describing the former space (originating in~\cite{W2}) is intended to
convey the fact that these metrics take a standard form on~a~much
larger region than typical metrics in
$\Riem_{\mathcal{S}_1}^{\psc}(X_1, \p X_1)_{h^{\beta}_1}$ and are thus
``Extra-standard''.

A typical element of $\Riem_{\mathcal{S}_1}^{\psc}(X_1, \p
X_1)_{h^{\beta}_1}$ (in the case when $L$ is a point and so $\p
X_1=\beta X_1$) is depicted in the upper right of
Figure~\ref{thmAdiagram-1}. Showing that, in the case when $L$ is a
point, a compact family of metrics in
$\Riem_{\mathcal{S}_1}^{\psc}(X_1, \p X_1)_{h^{\beta}_1}$ could be
continuously moved to a compact family of extra standard metrics in
$\Riem_{\mathcal{S}_1-\Estd}^{\psc}(X_1,\p X_1)_{h^{\beta}_1}$,
without moving already extra-standard metrics out of that space, is
the most technically difficult part of the proof of~\cite[Theorem~A]{W2}. This is done in~\cite[Section~6.6]{W2}.

In our case, $L$ is not simply a point. All of the adjustment takes
place off of $X$ and so we restrict our attention to $X_1\setminus
X$. Once again however, the argument in~\cite[Theorem~A]{W2} involves
continuous deformations of compact families of psc-metrics on the
base. Thus, applying Lemma~\ref{pscliftlemma}, for the one admissible
horizontal distribution $\bar{\mathcal{H}}$ and a sufficiently small
scaling constant $\bar{\tau}_{\min}>0$, all such (compact)
deformations of psc-metrics lift to appropriate deformations on the
total space $X_1\setminus X$. This completes the proof of Theorem~\ref{thm-fb-sing}.

\section{Proof of Theorem~\ref{thmA}}
{We now come to the proof of Theorem~\ref{thmA}. We~noted earlier (see Remark~\ref{Part(i)A}) that part~$(i)$ of~Theorem~\ref{thmA} follows from the main
 result of Chernysh in~\cite{Ch}. Thus, our focus is on part~$(ii)$.}

Let $X_{\Sigma}=X\cup_{\p X} -N(\beta X)$ be a pseudo-manifold as
above, where $X$ is a manifold with boundary $\p X$, the total space
of a bundle, $p\colon \p X\to \beta X$ with fibre $L$. This bundle has a~structure map $f\colon\beta X\to BG$. Recall that $G$ is a subgroup of the
isometry group of the simple link~$(L, g_{L})$. We~begin by recalling
the restriction map
\begin{gather*}
 \mathrm{res}\colon\ \Riem^{\psc}(X, \p X) \to \Riem^{\psc}(\p X),\qquad
\mathrm{res}\colon\ g \mapsto g|_{\p X}.
\end{gather*}
As we mentioned earlier, from work of Chernysh~\cite{Ch} and
Ebert--Frenck~\cite{EF}, we know that this map is a Serre fibre bundle.

Now we consider two pseudo-manifolds $X_{\Sigma} = X\cup_{\p X}
-N(\beta X)$ and $X_{\Sigma,1} = X_1\cup_{\p X} -N(\beta X_1)$, where
$X_1= X\cup_{\p X} Z$, and the manifold $Z$ is given by an elementary
psc-bordism $\bar B\colon \beta X \bord \beta X_1$ and a structure map
$\bar f\colon \bar B \to BG$, so that $f=\bar f|_{\beta X}$ and $f_1=\bar
f|_{\beta X_1}$. Namely, the manifold $Z$ is a total space of the
following smooth bundle:
\begin{gather*}
 \begin{diagram}
 \setlength{\dgARROWLENGTH}{1.95em}
\node{Z}
 \arrow{e,t}{\hat{\bar f}}
 \arrow{s,l}{p}
\node{E(L)}
 \arrow{s}
 \\
\node{\bar B}
 \arrow{e,t}{\bar f}
\node{BG.}
\end{diagram}
\end{gather*}
Consider a pair of psc-submersion metrics, $g^{\p}_0\in
\Riem_{\mathcal{S}}^{\psc}(\p X)$ and $g^{\p}_1\in
\Riem_{\mathcal{S}_1}^{\psc}(\p X_1)$, where $\mathcal{S}_1$ is the
analogue of the map $\mathcal{S}$ for the bundle $\p X_1\to \beta
X_1$. Thus, $g^{\p}_0$ and $g^{\p}_1$ are psc-submersion metrics over
$\beta X$ (respectively, over $\beta X_1$), with fibre metrics $\tau_0
g_L$ and $\tau_1 g_L$ for some constants $\tau_0, \tau_1>0$. We~denote by $h_0^{\beta}\in \Riem^{\psc}(\beta X)$ and $h_1^{\beta}\in
\Riem^{\psc}(\beta X_1)$, the respective base metrics
$\mathcal{B}\big(g^{\p}\big)=h^{\beta}=:h_0^{\beta}$ and
$\mathcal{B}_1\big(g^{\p}_1\big)=:h_1^{\beta}$.

Now we notice that the spaces $\Riem_{\mathcal{S}}^{\psc}(X, \p
X)_{h^{\beta}_0}$ and $\Riem_{\mathcal{S}}^{\psc}(X_1, \p
X_1)_{h^{\beta}_1}$ coincide with the fibres
\begin{gather*}
 \Riem_{\mathcal{S}}^{\psc}(X, \p X)_{h^{\beta}_0}= \mathrm{res}^{-1}_0\big(g_0^{\p}\big),\qquad
 \Riem_{\mathcal{S}}^{\psc}(X_1, \p X_1)_{h^{\beta}_1}= \mathrm{res}^{-1}_1\big(g_1^{\p}\big),
\end{gather*}
of the corresponding restriction maps:
\begin{gather*}
 \mathrm{res}_0\colon\ \Riem^{\psc}(X,\p X) \to \Riem^{\psc}(\p X), \qquad
 \mathrm{res}_1\colon\ \Riem^{\psc}(X_1,\p X_1) \to \Riem^{\psc}(\p X_1).
\end{gather*}
We consider the
inclusion map
\begin{gather*}
 \Riem_{\mathcal{S}}^{\psc}(\p X)\hookrightarrow \Riem^{\psc}(\p X).
\end{gather*}
Now, by definition, we obtain the space $\Riem^{\psc}(X_{\Sigma})$ as
a pull-back in the following diagram
\begin{gather*}
 \begin{diagram}
 \setlength{\dgARROWLENGTH}{1.95em}
\node{\Riem^{\psc}(X_{\Sigma})}
 \arrow{e,t}{\mathrm{res}_{\Sigma}}
 \arrow{s,l}{r}
\node{\Riem_{\mathcal{S}}^{\psc}(\p X)}
 \arrow{s,l,J}{}
 \\
\node{\Riem^{\psc}(X,\p X)}
 \arrow{e,t}{\mathrm{res}}
\node{\Riem^{\psc}(\p X),}
\end{diagram}
 \end{gather*}
 where the right vertical map is inclusion. The
 left vertical map sends a psc-metric on $X_{\Sigma}$ to its
 restriction to $X$.

There are two more maps we must define. Recall we have a projection map,
\begin{gather*}
\mathcal{B}\colon\ \Riem_{\mathcal{S}}^{\psc}\to \Riem^{\psc}(\beta X),
\end{gather*}
sending each submersion metric to its base. In Lemma~\ref{lifthe} we
showed that $\mathcal{B}$ forms part of a~homotopy equivalence. In
particular the lifting map~\eqref{SliftplusH}
\begin{gather*}
\mathcal{S}_{\mathcal{H}}^{\psc}\colon\ \Riem^{\psc}(\beta
X) \to \Riem_{\mathcal{S}}^{\psc},
\end{gather*}
for any choice of horizontal
distribution, $\mathcal{H}$, on the bundle $p\colon\p X\to \beta X$, is a
homotopy inverse. Composing, for some choice of $\mathcal{H}$, the map
$\mathcal{S}_{\mathcal{H}}^{\psc}$ with the inclusion map above yields
the map
\begin{gather*}
\pi_{\mathcal{H}}\colon \ \Riem^{\psc}(\beta
X)\xrightarrow[]{\mathcal{S}_{\mathcal{H}}^{\psc}}
\Riem_{\mathcal{S}}^{\psc} (\p X)\rh \Riem^{\psc}(\p X).
\end{gather*}
\noindent The restriction map $r\colon \Riem^{\psc}(X_{\Sigma})
\xrightarrow[]{} \Riem^{\psc}(X, \p X)$ above is in fact a
homeomorphism onto its image, $\Riem_{\mathcal{S}}^{\psc}(X, \p X)$,
the space of well adapted psc-metrics on $X$. Replacing the codomain
$\Riem^{\psc}(X, \p X)$ with $\Riem_{\mathcal{S}}^{\psc}(X, \p X)$, we
then compose $r$ with the restrictions $\mathrm{res}$ and
$\mathcal{B}$ to obtain the map
\begin{gather*}
{\mathrm{res}}_{\Sigma}\colon\ \Riem^{\psc}(X_{\Sigma}) \xrightarrow[]{r}
\Riem_{\mathcal{S}}^{\psc}(X, \p X) \xrightarrow[]{\mathrm{res}}
\Riem_{\mathcal{S}}^{\psc}(\p
X)\xrightarrow[]{\mathcal{B}}\Riem^{\psc}(\beta X).
\end{gather*}
As $r$ (as deployed above) is now a homeomorphism, $\mathrm{res}$
is a Serre fibre bundle and ${\mathcal{B}}$ is a~homotopy equivalence,
we conclude that the composition map, $\mathrm{res}_{\Sigma}$, is a
Serre fibre bundle. Recalling that the spaces $\Riem^{\psc}(X,\p
X)_{g_0^{\p}}=\mathrm{res}_{0}^{-1}\big(g_0^{\p}\big)$ and
$\Riem_{\mathcal{S}}^{\psc}(X,\p X)_{h_0^{\beta}}$ are in fact the
same, we obtain the following diagram of fibre bundles:
{\samepage\begin{gather*}
 \begin{diagram}
 \setlength{\dgARROWLENGTH}{1.95em}
\node{\Riem_{\mathcal{S}}^{\psc}(X,\p X)_{h_0^{\beta}}}
 \arrow{e,t}{i_{\Sigma}}
 \arrow{s,l}{=}
\node{\Riem^{\psc}(X_{\Sigma})}
 \arrow{e,t}{{\mathrm{res}}_{\Sigma}}
 \arrow{s,l}{r}
\node{\Riem^{\psc}(\beta X)}
 \arrow{s,l}{\pi_{\mathcal{H}}}
 \\
\node{\Riem^{\psc}(X,\p X)_{g_0^{\p}}}
 \arrow{e,t,J}{}
\node{\Riem^{\psc}(X,\p X)}
 \arrow{e,t}{\mathrm{res}}
\node{\Riem^{\psc}(\p X).}
\end{diagram}
\end{gather*}
Here, the top left horizontal map $i_{\Sigma}$ is the metric attaching
map defined in~\ref{metricattach}.

}

Let $\big(\bar B,\bar h_{\beta}\big)\colon \big(\beta X,h_{0}^{\beta}\big) \bord \big(\beta
X_1,h_1^{\beta}\big)$ be an elementary psc-bordism (with $p, q\geq 2$),
which is given together with a map $\bar f\colon \bar B \to BG$ such that
$f=\bar f|_{\beta X}$ and $f_1=\bar f|_{\beta X_1}$. In turn, the
psc-bordism $\big(\bar B,\bar h_{\beta}\big)$ determines a corresponding
psc-bordism $\big(Z,\bar g^{\p}\big)\colon \big(\p X, g^{\p}_0\big)\bord \big(\p X_1,
g^{\p}_1\big)$ by means of Lemma~\ref{pscliftlemma}. In particular, the
psc-submersion metrics $g^{\p}_0$ and $g^{\p}_1$ have base metrics
$h_{0}^{\beta}$ and $h_{1}^{\beta}$. Theorem~\ref{thm-boundary} and
Theorem~\ref{thm-fb-sing} give us the following homotopy equivalences:
\begin{gather*}
\mu_{\bar B,\bar g^{\beta}}\colon\ \Riem^{\psc}(\beta X)
\stackrel{\simeq}{\longrightarrow}
\Riem^{\psc}(\beta X_1),\nonumber
\\
\mu_{Z,\bar g^{\beta}}\colon\ \Riem^{\psc}(\p X)
\stackrel{\simeq}{\longrightarrow}
\Riem^{\psc}(\p X_1),\nonumber
\\
\mu_{Z,\bar g^{\p}}\colon\ \Riem^{\psc}(X,\p X)_{g^{\p}}
\stackrel{\simeq}{\longrightarrow}
\Riem^{\psc}(X_1,\p X_1)_{g^{\p}_1},\nonumber
\\
\mu_{\bar B, \bar f, \bar h^{\beta}}\colon\ \Riem_{\mathcal{S}}^{\psc}(X, \p X)_{h^{\beta}}
\stackrel{\simeq}{\longrightarrow}
\Riem_{\mathcal{S}_1}^{\psc}(X_1, \p X_1)_{h^{\beta}_1}.
\end{gather*}
We obtain the following commutative diagram:
\begin{gather*}
\xymatrix @!=3pc {
 & \Riem_{\mathcal{S}}^{\psc}(X, \p X)_{h^{\beta}} \ar[rr]^-{i_{\Sigma}}
\ar@{->}[dd]|!{[d];[d]}\hole
\ar[dl]_{\mu_{\bar B, \bar f, \bar h^{\beta}}\!\!\!\!}
& & \Riem^{\psc}(X_{\Sigma}) \ar@{->}[dd] |!{[d];[d]}\hole
\ar@{->}[rr]^{\mathrm{res}_{\Sigma}}\ar[dl]_*+[o][F-]{\simeq} 
&& \Riem^{\psc}(\beta X) \ar@{->}[dd]^(0.4){\pi_{\mathcal{H}_0}}\ar[dl]_{\mu_{\bar B,\bar g^{\beta}}\!\!\!\!}
\\
 \Riem_{\mathcal{S}_1}^{\psc}(X_1, \p X_1)_{h^{\beta}_1} \ar@{->}[dd]
 \ar@{->}[rr]^{\ \ \ \ \ i_{\Sigma}} && \Riem^{\psc}(X_{\Sigma,1}) \ar@{->}[dd]
 \ar@{->}[rr]^{\ \ \ \ \ \ \ \mathrm{res}_{\Sigma}} &&\Riem^{\psc}(\beta X_1)
 \ar@{->}[dd]^(0.4){\pi_{\mathcal{H}_1}}
\\
& \Riem^{\psc}(X,\p X)_{g^{\p}}
\ar@{->}[rr]^{\ \ \ i} |!{[r];[r]}\hole \ar[dl]_{\mu_{Z,\bar g^{\p}}\!\!\!\!}
& & \Riem^{\psc}(X,\p X) \ar@{->}[rr]^(0.4){\mathrm{res}}|!{[r];[r]}\hole
\ar[dl]_*+[o][F-]{\simeq}&&
\Riem^{\psc}(\p X) \ar[dl]_{\mu_{Z,\bar g^{\beta}}\!\!\!\!\!}
\\
\Riem^{\psc}(X_1,\p X_1)_{g^{\p}}
\ar@{->}[rr]^{i} & & \Riem^{\psc}(X_1,\p X_1) \ar@{->}[rr]^{\mathrm{res}}&&
\Riem^{\psc}(\p X_1)
}
\end{gather*}
where all horizontal rows are Serre fibre bundles, and the maps
$\pi_{\mathcal{H}_0}$ and $\pi_{\mathcal{H}_1}$ are defined as above.
Recall in particular that the left most vertical arrows (front and
back) represent equality. Commutativity of this diagram makes it
evident that the circled maps above are weak homotopy
equivalences. This proves Theorem~\ref{thmA}. \hfill $\Box$
\section{Some further developments}
In this section we would like to emphasize that recent results
concerning homotopy groups of the spaces, $\Riem^{\psc}(M)$, of
psc-metrics on a smooth closed manifold $M$ (see
\cite{BER-W,E-RW2,E-RW3,P2,P1}) could be applied directly and
indirectly to the case of manifolds with $(L,G)$-fibered
singularities. In~particular, we would like to attract the attention
of topologically-minded experts to relevant conjectures and results
from the recent work~\cite{BPR,BPR2}.

\subsection{Index-difference map}
We mentioned earlier that the homotopy-invariance of various spaces of
psc-metrics is a crucial property in helping detect their non-trivial
homotopy groups. With this in mind, there is a~secondary index
invariant, the index-difference map
\begin{gather}\label{eq1}
\mathsf{inddiff}_{g_0}\colon\ \Riem^{\psc}(M)\to \Omega^{\infty+n+1}\mathbf{KO},
\end{gather}
which is defined as follows. Let $g_0\in \Riem^{\psc}(M)$ be a base
point. Then for any psc-metric $g$ on $M$, there is an interval
$g_t=(1-t)g_0+tg$ of metrics such that a corresponding curve of the
Dirac operators $D_{g_t}$ starts and ends at the subspace
$\big(\mathbf{Fred}^n\big)^{\times}\subset \mathbf{Fred}^n$ of invertible
Dirac operators. Since the subspace $\big(\mathbf{Fred}^n\big)^{\times}$ is
contractible, the curve $D_{g_t}$ is a loop in the space
$\mathbf{Fred}^n$ of~all Dirac operators. This space, in turn, is
homotopy equivalent to the loop space $\Omega^{\infty+n}\mathbf{KO}$
representing a shifted $KO$-theory,
i.e., $\pi_q\big(\Omega^{\infty+n}\mathbf{KO}\big)=KO_{n+q}$. Thus the curve
$D_{g_t}$ gives an~element in $\Omega^{\infty+n+1}\mathbf{KO}$,
well-defined up to homotopy, to determine the map~\eqref{eq1}.

\begin{Theorem}[Botvinnik--Ebert--Randal-Williams~\cite{BER-W}, and Perlmutter~\cite{P2,P1}] 
Assume $M$ is a spin
 manifold with $\dim M \geq 5$ and $\Riem^{\psc}(M)\neq \varnothing$
 with a base point $g_0\in \Riem^{\psc}(M)$. Then the index-diffence
 map~\eqref{eq1} induces a non-trivial homomorphism in the
 homotopy groups
\begin{gather*}
 (\mathsf{inddiff}_{g_0})_*\colon\ \pi_q(\Riem^{\psc}(M))\to KO_{q+n+1},
\end{gather*}
when the target group $KO_{q+n+1}$ is non-trivial.
\end{Theorem}
\subsection{Results and conjectures}
The reader should note that much is also known about the spaces of psc-metrics for non-simply
connected manifolds; see~\cite{E-RW2,E-RW3}. We~will however return to the
same examples we considered above. We~have the following conjectures
concerning examples (1) and (2):
\begin{Conjecture}
 Let $X_{\Sigma}$ be a spin $(\left<k\right>\!\fb)$-manifold. Assume
 $\dim X\geq 7$ and $X$ and $\beta X\neq \varnothing $ are
 simply-connected and $\Riem^{\psc}(X_{\Sigma})\neq \varnothing$ with a
 base point $g_0\in \Riem^{\psc}(X_{\Sigma})$. Then there is an
 index-difference map
\begin{gather*}
 \mathsf{inddiff}_{g_0}^{\left<k\right>}\colon\ \Riem^{\psc}(X_{\Sigma})\to
 \Omega^{\infty+n+1}\mathbf{KO}^{\left<k\right>},
\end{gather*}
which induces a non-trivial homomorphism in the homotopy groups
\begin{gather*}
 \big(\mathsf{inddiff}_{g_0}^{\left<k\right>}\big)_*\colon\
 \pi_q(\Riem^{\psc}(X_{\Sigma}))\to KO^{\left<k\right>}_{n+q+1}
\end{gather*}
when the target group $KO^{\left<k\right>}_{n+q+1}$ $(KO$ with
$\Z_k$-coefficients$)$ is non-trivial.
\end{Conjecture}
\begin{Conjecture}
Let $X_{\Sigma}$ be a spin manifold with $(\eta\fb)$-singularity of
dimension $n\geq 9$. Assume $\beta X\neq \varnothing$, and
$\Riem^{\psc}(X_{\Sigma})\neq \varnothing$ with a base point $g_0\in
\Riem^{\psc}(X_{\Sigma})$. Then there is an index-difference map
\begin{gather*}
 \mathsf{inddiff}_{g_0}^{\eta\fb}\colon\ \Riem^{\psc}(X_{\Sigma})\to
 \Omega^{\infty+n+1}\mathbf{KO}^{\eta\fb},
\end{gather*}
which induces a non-trivial homomorphism in the homotopy groups
\begin{gather*}
 \big(\mathsf{inddiff}_{g_0}^{\eta}\big)_*\colon\
 \pi_q(\Riem^{\psc}(X_{\Sigma}))\to KO^{\eta\fb}_{q+n+1}
\end{gather*}
when the target group $KO^{\eta\fb}_{q+n+1}= KO_{q+n+1}(\CP^{\infty})$
is non-trivial.
\end{Conjecture}
It turns out that the above examples (3) and (4) (and many others; see
\cite{BPR}) lead to particular results concerning the homotopy groups
of the spaces $\Riem^{\psc}(X_{\Sigma})$. Let $X_{\Sigma}=X\cup_{\p X}
-N(\beta X)$ be a spin manifold with $(L,G)$-singularities. Let $g\in
\Riem^{\psc}(X_{\Sigma})$ be a well-adapted metric. Then $g$ determines
the metrics $g_{\p X}\in \Riem^{\psc}(\p X)$ and $g_{\beta X}\in
\Riem^{\psc}(\beta X)$ such that the bundle $\p X \to \beta X$ is a
Riemannian submersion. We~fix the metric $g_{\beta X,0}$.
This gives rise to a Serre fibre bundle
\begin{gather*}
\mathrm{res}_{\Sigma}\colon\ \Riem^{\psc}(X_{\Sigma})\to \Riem^{\psc}(\beta X)
\end{gather*}
with fibre $\Riem^{\psc}(X_{\Sigma})_{g_{\beta X},0}$, where
$\Riem^{\psc}(X_{\Sigma})_{g_{\beta X}}$ is the space of all metrics
$g\in \Riem^{\psc}(X_{\Sigma})$ which restrict to $g_{\beta X,0}$ on
$\Riem^{\psc}(\beta X)$. Since the metric $g_{\p X,0}$ on $\p X$ is
determined by the metric $g_{\beta X,0}$, the fibre
$\Riem^{\psc}(X_{\Sigma})_{g_{\beta X},0}$ coincides with the space
$\Riem^{\psc}(X)_{g_{\p X},0}$. Here is the result we need:

\begin{Theorem}[{see~\cite[Theorem~6.1]{BPR2}}]
Let $X_\Sigma$ be an $(L,G)$-fibred compact pseudo-manifold with $L$ a
simply connected homogeneous space of a compact semisimple Lie group.
Assume $\Riem^{\psc}(X_{\Sigma})\neq \varnothing$. Then there exists a
section $s\colon \mathcal{R}^{\psc}(\beta X)\to \mathcal{R}^{\psc}
(X_\Sigma)$ to $\mathrm{res}_{\Sigma}$. In particular, there is a
split short exact sequence:
\begin{gather*}
 0 \to \pi_q(\mathcal{R}_w^+ (X_\Sigma)_{g_{\beta
 X}})\xrightarrow{i_*} \pi_q(\mathcal{R}_w^+
 (X_\Sigma))\xrightarrow{(\mathrm{res}_{\Sigma})_*}\pi_q(\mathcal{R}^+(\beta
 X))\to 0 , \qquad q=0,1,\ldots.
\end{gather*}
\end{Theorem}
Here is one of the conclusions we would like to emphasize:
\begin{Corollary}[{see~\cite[Corollary 6.7]{BPR}}]
{\sloppy Let $X_\Sigma$ be an
 $(L,G)$-fibred compact pseudo-manifold with~$L$ a simply connected
 homogeneous space of a compact semisimple Lie group, and
 \mbox{$n-\ell-1\geq 5$}, where $\dim X=n$, $\dim L =\ell$. Let $g_0\in
 \Riem^{\psc}(X_{\Sigma})\neq \varnothing$ be a base point giving
 corresponding base points, the metrics $g_{\beta X,0}\in \Riem^{\psc}(\beta
 X)$, $g_{\p X,0}\in \Riem^{\psc}(\p X)$ and $g_{X,0}\in \Riem^{\psc}(X)_{g_{\p
 X,0}}$.

 }

If $M_\Sigma$ is spin and
simply connected, then we have the following commutative diagram:
\begin{gather*}
 \begin{diagram}
 \setlength{\dgARROWLENGTH}{1.95em}
\node{0\to \pi_q\Riem^{\psc}(X_\Sigma)_{g_{\beta X,0}}}
 \arrow{e,t}{j_*}
 \arrow{s,l}{\mathrm{inddiff}_{g_{X,0}}}
\node{\pi_q\mathcal{R}^{\psc} (X_\Sigma)}
 \arrow{e,t}{(\mathrm{res}_{\Sigma})_*}
 \arrow{s,l}{\mathrm{inddiff}_{g_0}}
\node{\pi_q\mathcal{R}^{\psc}(\beta X)\to 0}
 \arrow{s,l}{\mathrm{inddiff}_{g_{\beta X,0}}}
 \\
\node{\!\!\! 0\longrightarrow KO_{q+n+1}}
 \arrow{e}
\node{KO_{q+n+1}\oplus KO_{q+n-\ell} }
 \arrow{e}
\node{KO_{q+n-\ell}\to 0,}
\end{diagram}
\end{gather*}
where the homomorphisms $\mathrm{inddiff}_{g_{X,0}}$ and
$\mathrm{inddiff}_{g_{\beta X,0}}$ are both nontrivial whenever the
target groups are. In particular, the homomorphism
\begin{gather*}
 \mathrm{inddiff}_{g_0}\colon\ \pi_q\mathcal{R}^{\psc} (X_\Sigma) \to
 KO_{q+n+1}\oplus KO_{q+n-\ell}
\end{gather*}
is surjective rationally and surjective onto the torsion
of $KO_{q+n+1}\oplus KO_{q+n-\ell}$.
\end{Corollary}
There are much more general results concerning the homotopy groups of
the space $\pi_q\mathcal{R}^{\psc} (X_\Sigma)$ if $X_{\Sigma}$ is not
simply-connected; see~\cite[Section 6]{BPR}.

\section{Appendix: cone metrics}
\subsection[Cone metric of (L,gL) when gL
 has constant non-negative scalar curvature]%
{Cone metric of $\boldsymbol{(L, g_{L})}$ when $\boldsymbol{g_L}$
 has constant non-negative \\scalar curvature}
 Let $L$ be smooth closed
manifold of dimension $l$ and $g_{L}$, a metric on $L$ with constant
non-negative scalar curvature. We~wish to extend $g_{L}$ to a scalar
flat metric on the cone $C(L)$. The resulting metric will be denoted
$g_{C(L)}$ and is constructed as follows. Consider the warped product
metric ${\rm d}t^{2}+\phi(t)^{2}g_{L}$ on the cylinder $(0,b)\times L$ for
some smooth function $\phi\colon [0,b]\rightarrow (0,\infty)$. The scalar
curvature, $s$, of such a metric is given by the formula
\begin{gather}\label{conewarp}
s = \frac{-4l}{(l+1)}\left(\frac{u''}{u}\right)+s_{g_L}u^{\frac{-4}{l+1}},
\end{gather}
where $l:=\dim{L}$, $u$ is the function satisfying
$u^{\frac{2}{l+1}}=\phi$ and $s_{g_L}$ is the scalar curvature of the
metric $g_L$.

Let us first consider the case when $s_{g_L}=0$. An easy calculation
shows $s=0$ precisely when
\begin{gather*}
u(t)=At+B,
\end{gather*} for various constants
$A,B$. Indeed, in the case when $A=0$, we obtain a Riemannian
cylinder. In our case, we set $A=1$, $B=0$ and $b=\frac{1}{2}$. This
gives rise to a metric
\begin{gather*}
{g_{C(L)}}={\rm d}t^{2}+ t^{\frac{4}{l+1}}g_{L},
\end{gather*}
on the cylinder
$\big[0,\frac{1}{2}\big]\times L$, which collapses at $t=0$. The result is a
scalar-flat metric (with a~singularity) on the cone $C(L)$ obtained
from $\big[0,\frac{1}{2}\big]\times L$ by collapsing $\{ 0 \}\times L$ to a
point.

Now consider the case that the constant $s_{g_L}>0$. In this case, we
set $\phi(t)=\frac{1}{c_L}t$, where the constant $c_L$
satisfies
\begin{gather*}
c_L=\sqrt{\frac{l(l-1)}{s_{g_L}}}.
\end{gather*}
Thus
$u=t^{\frac{l+1}{2}}$ and the resulting metric
\begin{gather*}
{g_{C(L)}}= {\rm d}t^{2}+\frac{1}{c_{L}^{2}}t^{2}g_{L},
\end{gather*}
has scalar curvature given by
\begin{gather}\label{scalarflatcone}
s_{g_{C(L)}}=\big(c_{L}^{2}s_{g_L}-l(l-1)\big)t^{-2} = 0,
\end{gather}
given our choice of $c_{L}$. Again, we may assume that
$t\in\big[0,\frac{1}{2}\big]$.
\begin{Remark}
Regarding the earlier setting when $g_L$ is scalar flat, we will
typically only be interested in the case when $L=S^{1}$ and so
$l=1$. Thus, metric formulae in both cases coincide.
\end{Remark}

\subsection{Attaching the cone to the cylinder}
In constructing certain metrics on $X_{\Sigma}$, it is necessary to
attach the above conical metric, ${g_{C(L)}}$, to a cylindrical metric
${\rm d}t^{2}+g_{L}$. This requires an intermediary ``attaching" metric to
transition smoothly between the two forms. To do this, we first
specify a smooth transition function,
$a\colon [0,1]\rightarrow\big[0,\frac{1}{2}\big]$, satisfying:
\begin{enumerate}\itemsep=0pt
\item[$(i)$] $a(t)=\frac{1}{2}+t$ when $t$ is near $0$,
\item[$(ii)$] $a(t)=1$ when $t$ is near $1$,
\item[$(iii)$] for all $t\in[0,1]$, $0\leq a'(t)\leq 1$ and $a''(t)\leq 0$.
\end{enumerate}
The existence of such functions is obvious. We~choose one and define
the attaching metric, ${g_{\att(L)}}$, on the cylinder $[0, 1]\times
L$ by the formula:
\begin{gather*}
 {g_{\att(L)}}=
\begin{cases}
 {\rm d}t^{2}+a^{\frac{4}{l+1}}(t)g_{L},& s_{g_L}=0,
 \\
 {\rm d}t^{2}+a^{2}(t)g_{L},& s_{g_L}>0.
\end{cases}
\end{gather*}
\begin{Proposition}\label{attachprop}
Attaching $\big(\big[0, \frac{1}{2}\big]\times L, g_{C(L)}\big)$ to $\big([0, 1]\times L,
g_{\att(L)}\big)$ by identifying $\big\{\frac{1}{2}\big\}\times L$ in~the former
with $\{0\}\times L $ in the latter determines a smooth metric of
non-negative scalar curvature $g_{\att(L)}\cup g_{C(L)}$.
\end{Proposition}
\begin{proof}
Smoothness is immediate from the properties of the transition function
$a$. The component, $g_{C(L)}$, has zero scalar curvature and that
$g_{\att(L)}$ has non-negative scalar curvature is easily seen from
the formula (\ref{conewarp}).
\end{proof}

\subsection*{Acknowledgments}
BB was partially supported by Simons collaboration grant 708183. We~thank the referees for their careful reading and thoughtful
comments. The second author would like to thank David Wraith for some
productive conversations.

\pdfbookmark[1]{References}{ref}
\LastPageEnding


\begin{thebibliography}{99}
\footnotesize\itemsep=0pt

\bibitem{B-B}
B\'erard-Bergery L., Scalar curvature and isometry group, in Spectra of
 Riemannian Manifolds, Kaigai Publications, Tokyo, 1983, 9--28.

\bibitem{Besse}
Besse A.L., Einstein manifolds, \textit{Ergebnisse der Mathematik und ihrer
 Grenzgebiete~(3)}, Vol.~10, \href{https://doi.org/10.1007/978-3-540-74311-8}{Springer-Verlag}, Berlin, 1987.

\bibitem{B1}
Botvinnik B., Manifolds with singularities accepting a metric of positive
 scalar curvature, \href{https://doi.org/10.2140/gt.2001.5.683}{\textit{Geom. Topol.}} \textbf{5} (2001), 683--718,
 \href{https://arxiv.org/abs/math.DG/9910177}{arXiv:math.DG/9910177}.

\bibitem{BER-W}
Botvinnik B., Ebert J., Randal-Williams O., Infinite loop spaces and positive
 scalar curvature, \href{https://doi.org/10.1007/s00222-017-0719-3}{\textit{Invent. Math.}} \textbf{209} (2017), 749--835,
 \href{https://arxiv.org/abs/1411.7408}{arXiv:1411.7408}.

\bibitem{BPR}
Botvinnik B., Piazza P., Rosenberg J., Positive scalar curvature on simply
 connected spin pseudomanifolds, \href{https://arxiv.org/abs/1908.04420}{arXiv:1908.04420}.

\bibitem{BPR2}
Botvinnik B., Piazza P., Rosenberg J., Positive scalar curvature on spin
 pseudomanifolds: the fundamental group and secondary invariants,
 \href{https://arxiv.org/abs/2005.02744}{arXiv:2005.02744}.

\bibitem{BR}
Botvinnik B., Rosenberg J., Positive scalar curvature on manifolds with fibered
 singularities, \href{https://arxiv.org/abs/1808.06007}{arXiv:1808.06007}.

\bibitem{Ch}
Chernysh V., On the homotopy type of the space $\mathcal{R}^+(M)$,
 \href{https://arxiv.org/abs/math.GT/0405235}{arXiv:math.GT/0405235}.

\bibitem{EF}
Ebert J., Frenck G., The Gromov--Lawson--Chernysh surgery theorem,
 \href{https://arxiv.org/abs/1807.06311}{arXiv:1807.06311}.

\bibitem{E-RW2}
Ebert J., Randal-Williams O., Infinite loop spaces and positive scalar
 curvature in the presence of a fundamental group, \href{https://doi.org/10.2140/gt.2019.23.1549}{\textit{Geom. Topol.}}
 \textbf{23} (2019), 1549--1610, \href{https://arxiv.org/abs/1711.11363}{arXiv:1711.11363}.

\bibitem{E-RW3}
Ebert J., Randal-Williams O., The positive scalar curvature cobordism category,
 \href{https://arxiv.org/abs/1904.12951}{arXiv:1904.12951}.

\bibitem{GL}
Gromov M., Lawson Jr. H.B., The classification of simply connected manifolds of
 positive scalar curvature, \href{https://doi.org/10.2307/1971103}{\textit{Ann. of Math.}} \textbf{111} (1980),
 423--434.

\bibitem{P2}
Perlmutter N., Cobordism categories and parametrized Morse theory,
 \href{https://arxiv.org/abs/1703.01047}{arXiv:1703.01047}.

\bibitem{P1}
Perlmutter N., Parametrized Morse theory and positive scalar curvature,
 \href{https://arxiv.org/abs/1705.02754}{arXiv:1705.02754}.

\bibitem{St}
Stolz S., Simply connected manifolds of positive scalar curvature, \href{https://doi.org/10.2307/2946598}{\textit{Ann.
 of Math.}} \textbf{136} (1992), 511--540.

\bibitem{TW}
Tuschmann W., Wraith D.J., Moduli spaces of {R}iemannian metrics,
 \textit{Oberwolfach Seminars}, Vol.~46, \href{https://doi.org/10.1007/978-3-0348-0948-1}{Birkh\"auser Verlag}, Basel, 2015.

\bibitem{W0}
Walsh M., Metrics of positive scalar curvature and generalised {M}orse
 functions, {P}art~{I}, \href{https://doi.org/10.1090/S0065-9266-10-00622-8}{\textit{Mem. Amer. Math. Soc.}} \textbf{209} (2011),
 xviii+80~pages, \href{https://arxiv.org/abs/0811.1245}{arXiv:0811.1245}.

\bibitem{W1}
Walsh M., Cobordism invariance of the homotopy type of the space of positive
 scalar curvature metrics, \href{https://doi.org/10.1090/S0002-9939-2013-11647-3}{\textit{Proc. Amer. Math. Soc.}} \textbf{141}
 (2013), 2475--2484, \href{https://arxiv.org/abs/1109.6878}{arXiv:1109.6878}.

\bibitem{W3}
Walsh M., {$H$}-spaces, loop spaces and the space of positive scalar curvature
 metrics on the sphere, \href{https://doi.org/10.2140/gt.2014.18.2189}{\textit{Geom. Topol.}} \textbf{18} (2014), 2189--2243,
 \href{https://arxiv.org/abs/1301.5670}{arXiv:1301.5670}.

\bibitem{W2}
Walsh M., The space of positive scalar curvature metrics on a manifold with
 boundary, \textit{New York~J. Math.} \textbf{26} (2020), 853--930,
 \href{https://arxiv.org/abs/1411.2423}{arXiv:1411.2423}.

\end{thebibliography}
\end{document}